\documentclass[12pt,a4paper]{amsart}
\usepackage{amsaddr}
\usepackage{amsfonts}
\usepackage{amsthm}
\usepackage{amsmath}
\usepackage{xcolor}
\usepackage{amssymb}
\usepackage{mathtools} \usepackage{amscd}
\usepackage{dsfont} \usepackage[latin2]{inputenc}
\usepackage{t1enc}
\usepackage[mathscr]{eucal}
\usepackage{indentfirst}
\usepackage{graphicx}
\usepackage{graphics}
\usepackage{epic}
\numberwithin{equation}{section}
\usepackage[margin=2.9cm, marginparwidth=2.4cm]{geometry}
\usepackage{epstopdf} 
\usepackage{enumerate}
\usepackage{mathrsfs} \usepackage{cases}

\usepackage[colorlinks=true, pdfstartview=FitV, linkcolor=blue,
citecolor=blue, urlcolor=blue,pagebackref=false]{hyperref}

\usepackage{microtype}

\theoremstyle{plain}
\newtheorem{Th}{Theorem}[section]
\newtheorem{Lemma}[Th]{Lemma}

\newtheorem{Prop}[Th]{Proposition}

 \theoremstyle{definition}
\newtheorem{Def}[Th]{Definition}
\newtheorem*{Def*}{Definition}

\newtheorem{Rem}[Th]{Remark}
\newtheorem*{Rem*}{Remark}
\newtheorem{?}[Th]{Problem}

\newcommand{\E}{\mathbb{E} }
\renewcommand{\P}{\mathbb{P}}

\newcommand{\Ind}[1]{\mathds{1}_{\{ #1\}}}

\newcommand{\R}{\mathbb{R}}
\newcommand{\N}{\mathbb{N}}

\renewcommand{\d}{\mathrm{d}}
\newcommand{\la}{\langle}\newcommand{\ra}{\rangle}
\newcommand{\eps}{\epsilon}

\newcommand{\hone}{{h_1}}
\newcommand{\htwo}{{h_2}}

\newcommand{\bF}{\overline{F}}
\renewcommand{\div}{\mathrm{div}}
\renewcommand{\H}{\mathsf{H}}
\newcommand{\tz}{\tilde z}
\newcommand{\ty}{\tilde y}

\newcommand{\tX}{\widetilde{X}}
\newcommand{\cH}{\mathcal{H}}
\newcommand{\tU}{\widetilde{U}}

\usepackage{ulem}
\usepackage{cancel}

\begin{document}

\title{Hamilton--Jacobi equations for  nonsymmetric matrix inference}
\author{Hong-Bin Chen}
\begin{abstract}
    We study the high-dimensional limit of the free energy associated with the inference problem of a rank-one nonsymmetric matrix. The matrix is expressed as the outer product of two vectors, not necessarily independent. The distributions of the two vectors are only assumed to have scaled bounded supports. We bound the difference between the free energy and the solution to a suitable Hamilton--Jacobi equation in terms of two much simpler quantities: concentration rate of this free energy, and the convergence rate of a simpler free energy in a decoupled system. To demonstrate the versatility of this approach, we apply our result to  the i.i.d.\ case and the spherical case. By plugging in estimates of the two simpler quantities, we identify the limits and obtain convergence rates.
\end{abstract}
\address{Courant Institute of Mathematical Sciences, New York University}
\email{hbchen@cims.nyu.edu}
\date{\today}

\maketitle

\section{Introduction}
Recovering a matrix from a noisy observation is a basic problem in statistical inference. Our setting is in the high-dimensional regime. For $n\in\N$, let $m=m(n)\in\N$  be a function of $n$ satisfying $\lim_{n\to\infty}m(n)=\infty$.   
Let $X=(X_1,X_2,\dots,X_m)\in\R^m$ and $Y=(Y_1,Y_2,\dots, Y_n)\in\R^n$ be two random vectors with joint law $P^{X,Y}_n$, which are \textit{not} necessarily independent. We assume  that
\begin{align}\label{eq:support_X_Y}
    |X|\leq \sqrt{m},\quad |Y|\leq \sqrt{n}, \quad \text{a.s.\ }\forall n.
\end{align}
Let $N=N(n)=\sqrt{mn}$ be the geometric mean of the sizes $m$ and $n$. 

\smallskip

The noisy observation is given by 
\begin{align}\label{eq:Z_observe}
    Z=\sqrt{\frac{2t}{N}}XY^\intercal + W
\end{align}
where $2t\geq 0$ is interpreted as the signal-to-noise ratio, and $W=(W_{ij})$ is an $m\times n $ matrix with independent standard Gaussian entries. The goal of inference is to recover information about the matrix $XY^\intercal$ from the observation $Z$. 

\smallskip

Of particular interest here is to study the law of $(X,Y)$ conditioned on observing $Z$. For any bounded measurable function $f:\R^m\times \R^n\to\R$, using the Bayes' rule, we can obtain the following formula: 
\begin{align*}
    \E\big[f(X,Y)\big|Z\big] = \frac{\int_{\R^m\times \R^n}f(x,y)e^{\mathring H_n(t,x,y)}P^{X,Y}_n(\d x,\d y)}{\int_{\R^m\times \R^n}e^{\mathring H_n(t,x,y)}P^{X,Y}_n(\d x,\d y)}
\end{align*}
where $\mathring H_n(t,x,y)= \sqrt{\frac{2t}{N}}x \cdot Wy+\frac{2t}{N}xy^\intercal\cdot XY^\intercal -\frac{t}{N}|xy^\intercal |^2$ is the Hamiltonian associated with this model. Here and throughout the paper, the dot product between two matrices or vectors of the same size denotes the entry-wise inner product, namely, $A\cdot B = \sum_{i,j}A_{ij}B_{ij}$. 

\smallskip

As in problems of statistical mechanics, many properties of the system can be understood by investigating the behavior of the associated free energy. In this work, we are concerned with the limit of the free energy given below, as $n\to \infty$,
\begin{align*}
    \mathring F_n(t) = \frac{1}{N}\E\log \int_{\R^m\times\R^n}e^{\mathring H_n(t,x,y)}P^{X,Y}_n(\d x,\d y).
\end{align*}

\smallskip

This work follows the approach set forth in \cite{mourrat2018hamiltonjacobi, mourrat2019hamilton, mourrat2019parisi, mourrat2020extending, mourrat2020nonconvex}. We seek to identify the limit of an enriched version of the free energy $\mathring F_n$ as a solution to a suitable Hamilton--Jacobi equation. Let us introduce the enriched Hamiltonian, for $(t,h)\in [0,\infty)\times [0,\infty)^2$,
\begin{align}\label{eq:H_n_expression}
\begin{split}
    H_n(t,h,x,y)&= \sqrt{\frac{2t}{N}}x \cdot Wy+\frac{2t}{N}xy^\intercal\cdot XY^\intercal -\frac{t}{N}|xy^\intercal |^2\\
    &+ \sqrt{2\hone }U\cdot x + 2\hone X\cdot x - \hone |x|^2\\
    & +\sqrt{2\htwo}V\cdot y +2\htwo Y\cdot y - \htwo |y|^2, 
\end{split}
\end{align}
where $U=(U_1,U_2,\dots, U_m)\in\R^m$ and $V=(V_1,V_2,\dots, V_n)\in\R^n$ are independent standard Gaussian vectors. The free energy associated with this enriched Hamiltonian is 
\begin{align}\label{eq:def_F_n}
    F_n(t,h) = \frac{1}{N}\log \int e^{H_n(t,h,x,y)}P^{X,Y}_n(\d x,\d y),
\end{align}
and its expectation is written as $\bF_n(t,h)= \E F_n(t,h)$. 

\smallskip

Note that $\bF_n(0,\cdot)$ can be viewed as the free energy of a decoupled system, which is a much simpler object to analyze. In addition to assumption \eqref{eq:support_X_Y}, we also assume that $\bF_n(0,\cdot)$ converges to a function $\psi$. Due to the absence of coupled interaction in $\bF_n(0,\cdot)$, this assumption should, in general, be relatively easy to check. Moreover, explicit estimates on the speed of convergence should also be attainable. 

\smallskip

Under these assumptions, we will show that $\bF_n$ satisfies an approximate Hamilton--Jacobi equation, whose limiting equation is 
\begin{align}\label{eq:HJ_intro}
    \partial_t f - (\partial_{\hone} f)(\partial_{\htwo}f) = 0, \quad \text{in $[0,\infty)\times [0,\infty)^2$},
\end{align}
with initial condition $f(0,\cdot) =\psi$.
We adopt the notion of weak solutions advocated in \cite{mourrat2019hamilton}, for this notion allows a simpler way to study the convergence of $\bF_n$. 
Let us rewrite the nonlinear term in \eqref{eq:HJ_intro} as $(\partial_\hone f)(\partial_\htwo f)=\H(\nabla f)$ where $\H:\R^2\to\R$ is given by $\H(p)=p_1p_2$. Different from \cite{mourrat2018hamiltonjacobi} and \cite{mourrat2019hamilton} which studied the symmetric matrix inference problem, in the nonsymmetric setting, the function $\H$ is not convex. As in \cite{bardi1984hopf} and \cite{lions1986hopf}, the existence of solution
is ensured and can be expressed by a variational formula if the initial condition is convex. For the uniqueness, a partial convexity condition (see \eqref{item:4} of Definition~\ref{def:weak_sol}) and the nonnegativity of all entries in the Hessian of $\H$ are sufficient.

\smallskip

The inference problem of symmetric matrices, in the rank-one case or more general, has been extensively studied. We refer to \cite{lelarge2019fundamental} for a description of these results. For the nonsymmetric matrix inference as is concerned here, the problem has been studied in \cite{miolane2017fundamental} for the case where $X$ and $Y$ have i.i.d.\ entries. In particular, a variational formula for the limit of the free energy is obtained. Most recently, employing the adaptive interpolation method (introduced in \cite{barbier2019adaptive}), \cite{luneau2020high} established results in the case where $X$ and $Y$ are uniformly distributed on properly scaled spheres.  More generally, results have been extended to rank-one tensor inference in \cite{barbier2017layered} and \cite{kadmon2018statistical}. The inference of second-order matrix tensor product is recently studied in \cite{reeves2020information} using a more general adaptive interpolation method.

\smallskip

Our main contribution is to provide a different approach, which is conceptually simpler in some aspect.
The main result bounds the local $L^\infty_tL^1_h$ norm of $\bF_n - f$ in terms of two quantities: 1) the concentration rate of $F_n$ towards $\bF_n$ and 2) the convergence rate of $\bF_n(0,\cdot)$. These two quantities are much simpler to study, and there is a plethora of tools and techniques to obtain good estimates. Convergence in the local uniform topology can also be obtained as discussed in Remark~\ref{Rem:infinity_norm}. As an application of the main result, we estimate the two quantities for the i.i.d.\ case and the spherical case. We recovered known results on the limit of the free energy for these two cases simultaneously. In addition, convergence rates are also obtained.

\smallskip

\smallskip

The rest of the paper is organized as follows. In Section~\ref{section:setting}, we state the main results. We also include results for the i.i.d.\ case and spherical case. Then, we show $\bF_n$ satisfies an approximate Hamilton--Jacobi equation, and list a few basic estimates in Section~\ref{section:approx_hj_eqn}. In Section~\ref{section:HJ_eqn}, we define the notion of weak solutions, prove the uniqueness of solutions, and describe conditions for the existence of solutions. Section~\ref{section:cvg_fre_energy} contains the proof of the main results. Lastly, in Section~\ref{section:application}, we collect estimates needed to derive convergence results for the two special cases. In addition, we briefly describe possible modifications of the current approach to study the sparse model in Section~\ref{section:sparse}.

\smallskip

Throughout the paper, we write $\R_+= [0,\infty)$. The symbol $C$ denotes a positive absolute constant which may vary from instance to instance.

\subsection*{Acknowledgements}
I am grateful to Jean-Christophe Mourrat for introducing me to this subject and for many helpful insights. I would like to thank Jiaming Xia for helpful comments.

\section{Settings and main results}\label{section:setting}

\subsection{General setting}

We assume \eqref{eq:support_X_Y} and the existence of $\alpha>0$ such that
\begin{align}\label{eq:m/n_to_alpha}
    \lim_{n\to\infty}\frac{m(n)}{n}=\alpha>0.
\end{align}
Let us define the following quantities, for $M>0$, $n\in\N$ and $\psi:\R^2_+\to\R$,
\begin{align}
    K_{M,n}&= \bigg(\E\sup_{(t,h)\in[0,M]^3}\big|F_n-\bF_n|^2\bigg)^\frac{1}{2},\label{eq:def_K}\\
    L_{\psi,M,n}&= \sup_{h\in[0,M]^2}\big|\bF_n(0,h) - \psi(h)\big|\label{eq:def_L}.
\end{align}

\begin{Th}\label{Prop:general_cvg_bF}
Suppose that $X$ and $Y$ are independent for all $n$, and that there is a function $\psi:\R^2_+\to \R$ such that 
\begin{align}\label{eq:assumption_gen_cvg}
\bF_n(0,\cdot)\to \psi,\quad  \text{ pointwise as $n\to \infty$}.     
\end{align}
Then
\begin{enumerate}
    \item there is a unique weak solution $f$ to the Hamilton--Jacobi equation
\begin{align}\label{eq:general_2d_HJ}
\begin{cases}
\partial_t f - \big(\partial_\hone f\big)\big( \partial_\htwo f\big) = 0, &\quad \text{in } \R_+ \times \R_+^2,\\
f(0,\cdot) =\psi, &\quad \text{in } \R^2_+;
\end{cases}
\end{align}
and $f$ admits a variational representation known as the Hopf formula 
\begin{align}\label{eq:Hopf_2d}
    f(t,h) = \sup_{z\in\R^2_+}\inf_{y\in\R^2_+}\big\{z\cdot (h-y)+\psi(y) + tz_1z_2 \big\};
\end{align}

\item there is a constant $C>0$ such that the following holds for all $M\geq 1$ and all $n\in\N$:
\begin{align*}
    \sup_{t\in [0,M]}\int_{[0,M]^2}\big|\bF_n(t,h)-f(t,h)\big|\d h\leq CM^2\Big(L_{\psi,CM,n}+n^{-1}+ (K_{CM,n})^\frac{2}{3}+K_{CM,n}\Big).
\end{align*}
\end{enumerate}
\end{Th}

This theorem is a consequence of Theorem~\ref{thm:most_general_cvg} in a more general setting.

\begin{Rem}\label{Rem:X_Y_not_independent}
When $X$ and $Y$ are not independent, a similar result can still be obtained. To guarantee the existence of a weak solution, we need to further assume, for each $M\geq1$,
\begin{align*}
    \lim_{n\to\infty}L_{\psi, M,n}=0,\quad \lim_{n\to\infty}K_{M,n}=0.
\end{align*}
This result is recorded in Proposition~\ref{prop:cvg_for_non_independent}. 

\end{Rem}

\begin{Rem}\label{Rem:infinity_norm}
Assuming that the right hand side of the inequality in the second part of Theorem~\ref{Prop:general_cvg_bF} converges to $0$ as $n\to\infty$, we can also obtain local uniform convergence by utilizing the fact $f$ is Lipschitz and $\bF_n $ is Lipschitz uniformly in $n$ (see Definition \ref{def:weak_sol} and  \eqref{eq:grad_bF<C}). We briefly sketch the argument. Let $\xi:\R^2\to\R$ be a smooth radial bump function supported on the unit disk and satisfy $0\leq \xi \leq 1$ and $\int \xi =1$. For $\eps \in(0,2)$, set $\xi_\eps(x) = \eps^{-2}\xi(\eps^{-1}x)$. Notice that for any Lipschitz $g:\R^2\to \R$ and $A\subset \R^2$, we have 
\begin{align*}
    \|g\|_{L^\infty(A)}\leq \|g*\xi_\eps\|_{L^\infty(A)}+\|g-g*\xi_\eps\|_{L^\infty(A)}\leq \eps^{-2}\|g\|_{L^1(A_\eps)}+\eps \|\nabla g\|_{L^\infty(\R^2)},
\end{align*}
where $A_\eps$ is the $\eps$-neighbourhood of $A$. Extend $\bF_n(t,\cdot)$, $ f(t,\cdot)$ symmetrically to $\R^2$. For any compact $B\subset \R^2$, we choose $M$ large so that $B_1\subset [-M,M]^2$. Fix any $t\in[0,M]$ and set $g_n=\bF_n(t,\cdot)-f(t,\cdot)$. By our assumption at the beginning of this remark, we have $\lim_{n\to\infty}\|g_n\|_{L^1([-M,M]^2)}=0$. Then for large $n$, we have $\eps=(\|g_n\|_{L^1([-M,M]^2)}\|\nabla g_n\|_\infty)^\frac{1}{3}<1$. Plug this $\eps$ to the above display to see
\begin{align*}
    \|g_n\|_{L^\infty(B)}\leq C\|g_n\|_{L^1([-M,M]^2)}^\frac{1}{3}\|\nabla g_n\|_\infty^\frac{2}{3} \leq C\|g_n\|_{L^1([-M,M]^2)}^\frac{1}{3}.
\end{align*}
Here the last equality is due to the Lipschitzness of $f$ and $\bF_n$ (uniformly in $n$).

\end{Rem}

\subsection{Special cases}

We apply Theorem~\ref{Prop:general_cvg_bF} to the i.i.d.\ case and the spherical case. In Section \ref{section:application}, we prove these results by identifying $\psi$ in \eqref{eq:def_general_psi} and estimating $K_{M,n}$ and $L_{\psi,M,n}$ in these two cases. To quantify the rate of convergence in \eqref{eq:m/n_to_alpha}, we set
\begin{align}\label{eq:def_beta(n)}
    \beta(n)=\Big|\frac{m(n)}{n}-\alpha\Big|.
\end{align}

\subsubsection{The i.i.d.\ case}\label{section:iid}
Let $P^X_1$ and $P^Y_1$ be supported on $[-1,1]$. Recall that $m=m(n)$ is a function of $n$. For $ n \in \N$, we set
\begin{align*}
    P^{X,Y}_n = \big(P^X_1)^{\otimes m}\otimes\big(P^Y_1)^{\otimes n},
\end{align*}
where $\otimes$ denotes the product of measures.
Hence, $X$ and $Y$ have i.i.d.\ entries, distributed according to $P^X_1$ and $P^Y_1$, respectively.

Here is the result for the i.i.d.\ case. 

\begin{Prop}[Convergence for the i.i.d.\ case] \label{Prop:cvg_iid}
There is a constant $C>0$, such that the following holds for all $M\geq 1$ and $n\in\N$:
\begin{align*}
    \sup_{t\in [0,M]}\int_{[0,M]^2}\big|\bF_n(t,h)-f(t,h)\big|\d h \leq CM^3\Big(\beta(n)+\big(n^{-1}\log n\big)^\frac{1}{3}\Big).
\end{align*}
Here $f$ is the unique weak solution to \eqref{eq:general_2d_HJ} with $\psi$ given by
\begin{align}\label{eq:psi_iid}
    \psi(h) = \Big(\alpha m(1)\Big)^\frac{1}{2}\bF_1(0,\hone,0) + \Big(\alpha^{-1}m(1)\Big)^\frac{1}{2}\bF_1(0,0,\htwo).
\end{align}
\end{Prop}
In \eqref{eq:psi_iid}, $m(1)$ is the evaluation of $m=m(n)$ at $n=1$.

\subsubsection{The spherical case}\label{section:spherical}
For $k\in\N$, let $\mathscr{U}_k$ be the uniform measure on the centered sphere with radius $\sqrt{ k}$, denoted as $\sqrt{k}\mathbb{S}^{k-1}$. Consider the joint distribution of $X$ and $Y$ given by
\begin{align*}
    P^{X,Y}_n(\d x, \d y) = \mathscr{U}_m(\d x)\otimes \mathscr{U}_n(\d y).
\end{align*}
In particular, $X$ and $Y$ are independent. The result in this case is stated below.

\begin{Prop}[Convergence for the spherical case] \label{Prop:cvg_spherical}
There is a constant $C>0$, such that the following holds for all $M\geq 1$ and $n\in\N$:
\begin{align*}
    \sup_{t\in [0,M]}\int_{[0,M]^2}\big|\bF_n(t,h)-f(t,h)\big|\d h \leq CM^3\Big(\beta(n)+\big(n^{-1}\log n\big)^\frac{1}{3}\Big).
\end{align*}
Here $f$ is the unique weak solution to \eqref{eq:general_2d_HJ} with $\psi$ given by
\begin{align}\label{eq:psi_spherical}
    \psi(h) = \alpha \Big(\hone -\frac{\log(1+2\hone)}{2}\Big) + \alpha^{-1}\Big(\htwo -\frac{\log(1+2\htwo)}{2}\Big).
\end{align}
\end{Prop}

\section{Approximate Hamilton--Jacobi equations}\label{section:approx_hj_eqn}

In this section, we show that $\bF_n$ satisfies an approximate Hamilton--Jacobi equation as stated in Lemma~\ref{lemma:HJ}, under the notation and assumptions introduced in the previous two sections.

\begin{Lemma}\label{lemma:HJ}
For each $n\in\N$, the function $\bF_n:\R_+\times\R_+^2 \to \R$ satisfies
\begin{align*}
    \Big|\partial_t \bF_n - \big(\partial_\hone \bF_n\big)\big( \partial_\htwo \bF_n\big) \Big| \leq  \frac{1}{2N} \Delta\bF_n+\frac{1}{2}\E\big|\nabla (F_n- \bF_n)\big|^2 .
\end{align*}
\end{Lemma}

The Laplacian $\Delta$ and the gradient $\nabla$ are all carried out in the variable $h$. Note that the relation $\bF_n$ satisfies in this lemma bears a strong resemblance to the Hamilton--Jacobi equation \eqref{eq:general_2d_HJ}. The Laplacian term above can be seen as the vanishing viscosity. The last term has a flavor of concentration. These suggest that $\bF_n$ should satisfy \eqref{eq:general_2d_HJ} asymptotically.

\subsection{Proof of Lemma \ref{lemma:HJ}}
Let us introduce the notation $\la \cdot \ra$ for the Gibbs measure. For a measurable function $f:\R^m\times \R^n\to \R$, we write
\begin{align*}
    \la f(x,y) \ra = \big(\mathscr{Z}(t,h)\big)^{-1} \int_{\R^m\times \R^n}f(x,y)e^{H_n(t,h,x,y)}P^{X,Y}_n(\d x,\d y),
\end{align*}
with the normalizing factor
\begin{align*}
    \mathscr{Z}(t,h) = \int_{\R^m\times \R^n}e^{H_n(t,h,x,y)}P^{X,Y}_n(\d x,\d y).
\end{align*}
In view of $H_n$ defined in \eqref{eq:H_n_expression}, the probability measure $\la \cdot \ra$ is random, and depends on $t$ and $h$. For simplicity of notation, such dependence is suppressed. We also consider independent copies of $(x,y)$ with respect to this Gibbs measure. They are called ``replicas'' and denoted as $(x',y')$, $(x'',y'')$, etc. 

Two important tools in our computations are the Nishimori identity and the Gaussian integration by parts. 
The Nishimori identity can be stated as the following: for a bounded measurable $f:(\R^m\times \R^n)^2 \to\R$,
\begin{align}\label{eq:Nishimori}
    \E\la f(x,y,X,Y) \ra = \E \la f(x,y,x',y')\ra,
\end{align}
where $(x',y')$ is replica of $(x,y)$.
To show this, the key step is to verify the following identity based on Bayes' rule
\begin{align*}
    \la f(x,y)\ra = \E\big[ f(X,Y)|\mathcal{Z}\big]
\end{align*}
where, with $Z$ given in \eqref{eq:Z_observe},
\begin{align*}
    \mathcal{Z} = \big(Z,\ \sqrt{2h_1}X+U,\ \sqrt{2h_2}Y+V\big).
\end{align*}
An obvious extension of \eqref{eq:Nishimori} holds when more replicas are involved. In other words, the Nishimori identity allows us to replace one replica by $(X,Y)$ inside $\E\la \cdot \ra$. For more details, one can see \cite[Section 3.1]{mourrat2019hamilton} in a slightly different setting. 

The other ingredient,
the Gaussian integration by parts, is essentially the  fact that, for any smooth $f:\R\to \R$ with polynomial growth,
\begin{align*}
    \int_\R xf(x)e^{-x^2/2}\d x = \int_\R f'(x)e^{-x^2/2}\d x.
\end{align*}

\bigskip

Recall the definition of $F_n$ in \eqref{eq:def_F_n}. We start to compute the derivatives of $F_n$ and $\bF_n$. The first order derivatives of $F_n$ are given by
\begin{align}
    \partial_t F_n(t,h) &= \frac{1}{N}\bigg\la \frac{1}{\sqrt{2Nt}}y\cdot Wx+\frac{2}{N}xy^\intercal \cdot XY^\intercal -\frac{1}{N^2}|xy^\intercal|^2\bigg\ra,\label{eq:1st_der_F_t}\\
    \partial_\hone F_n(t,h) &= \frac{1}{N}\la \partial_\hone H_n\ra = \frac{1}{N}\bigg\la \frac{1}{\sqrt{2\hone}}U\cdot x + 2x\cdot X-|x|^2\bigg\ra,\label{eq:1st_der_F}\\
    \partial_\htwo F_n(t,h) &= \frac{1}{N}\la \partial_\htwo H_n\ra = \frac{1}{N}\bigg\la \frac{1}{\sqrt{2\htwo}}V\cdot y+ 2y\cdot Y-|y|^2\bigg\ra.\label{eq:1st_der_F_h_2}
\end{align}

After taking expectations, we apply the Nishimori identity and the Gaussian integration by parts to obtain
\begin{align}
    \partial_t \bF_n(t,h) &= \frac{1}{N^2}\E\la(x\cdot x')(y\cdot y')\ra,\label{eq:1st_der_bF_t}\\
    \partial_\hone \bF_n(t,h) &= \frac{1}{N}\E \la \partial_\hone H_n\ra =\frac{1}{N}\E\la x\cdot x'\ra,\label{eq:1st_der_bF}\\
    \partial_\htwo \bF_n(t,h) &= \frac{1}{N}\E \la \partial_\htwo H_n\ra =\frac{1}{N}\E\la y\cdot y'\ra.\label{eq:1st_der_bF_h_2}
\end{align}
Therefore, we have
\begin{align*}
    \partial_t \bF_n - \big(\partial_\hone \bF_n\big)\big( \partial_\htwo \bF_n \big)& = \frac{1}{N^2}\Big(\E\la(x\cdot x')(y\cdot y')\ra - \E\la x\cdot x'\ra \E\la y\cdot y'\ra \Big)\\
    & = \frac{1}{N^2}\E\Big\la (x\cdot x'- \E\la x\cdot x'\ra) (y\cdot y'- \E\la y\cdot y'\ra)\Big\ra.
\end{align*}
After an application of the Cauchy--Schwarz inequality, it becomes
\begin{align}\label{eq:HJ_eqn}
    \Big|\partial_t \bF_n - \big(\partial_\hone \bF_n\big)\big( \partial_\htwo \bF_n \big) \Big|\leq \frac{1}{2N^2}\Big(\E\la(x\cdot x'- \E\la x\cdot x'\ra)^2\ra + \E\la(y\cdot y'- \E\la y\cdot y'\ra)^2\ra\Big).
\end{align}

We will show the following upper bound
\begin{align}\label{eq:key_up_bdd_approx_HJ}
    \E\la(x\cdot x'- \E\la x\cdot x'\ra)^2\ra \leq N\partial^2_\hone \bF_n + N^2\E(\partial_\hone F_n - \partial_\hone \bF_n)^2. 
\end{align}
A similar upper bound for the second term on the right of \eqref{eq:HJ_eqn} can be obtained. Lemma~\ref{lemma:HJ} follows from these two upper bounds and \eqref{eq:HJ_eqn}.

\subsubsection{Proof of \eqref{eq:key_up_bdd_approx_HJ}}
Using the expression of $\partial_\hone F_n$ in \eqref{eq:1st_der_F}, we can calculate
\begin{align}\label{eq:2nd_der_F}
    N \partial^2_\hone F_n & = \la (\partial_\hone H_n)^2\ra - \la \partial_\hone H_n\ra^2 -\frac{1}{(2\hone)^\frac{3}{2}}\la U\cdot x\ra
\end{align}
Apply the Gaussian integration by parts to obtain
\begin{align}\label{eq:2nd_der_bF}
    N \partial^2_\hone \bF_n= \E\la (\partial_\hone H_n)^2\ra - N^2\E(\partial_\hone F_n)^2 -\frac{1}{2\hone}\E\la|x|^2\ra + \frac{1}{2\hone}\E|\la x\ra|^2\ra.
\end{align}
The key computation is the following
\begin{align}\label{eq:key_lower_bd_approx_hj}
    \E\la (\partial_\hone H_n)^2\ra \geq \E\la (x\cdot x')^2\ra  + \frac{1}{2\hone}\E\la|x|^2\ra.
\end{align}
This will be done slightly later. Now insert \eqref{eq:key_lower_bd_approx_hj} into \eqref{eq:2nd_der_bF} to see
\begin{align*}
    N \partial^2_\hone \bF_n \geq \E\la (x\cdot x')^2\ra -   N^2\E(\partial_\hone F_n)^2. 
\end{align*}
Finally, by \eqref{eq:1st_der_bF}, we have
\begin{align*}
\E\la(x\cdot x'- \E\la x\cdot x'\ra)^2\ra = \E \la (x\cdot x')^2\ra - (\E\la x\cdot x'\ra)^2 = \E \la (x\cdot x')^2\ra - N^2 (\partial_\hone \bF_n)^2.
\end{align*}
From the above two displays, we can deduce \eqref{eq:key_up_bdd_approx_HJ}.

\bigskip

Lastly, let us derive \eqref{eq:key_lower_bd_approx_hj}. Using the expression of $\partial_\hone H_n$ in  \eqref{eq:1st_der_F}, we have
\begin{align}\label{eq:E(partial_H)^2}
\begin{split}
    &\E\la (\partial_\hone H_n)^2\ra = \E \bigg\la \bigg( \frac{1}{\sqrt{2\hone}}U\cdot x + 2x\cdot X-|x|^2\bigg)^2\bigg\ra=\\
    &\E \bigg\la \frac{1}{2\hone }(U\cdot x)^2  + 4(x\cdot X)^2+|x|^4 +\frac{4}{\sqrt{2\hone}}(U\cdot x)(x\cdot X)-\frac{2}{\sqrt{2\hone}}(U\cdot x)|x|^2-4(x\cdot X)|x|^2 \bigg\ra
\end{split}
\end{align}
Let us rewrite the first term in the above display as
\begin{align*}
    &\E\bigg\la\frac{1}{2\hone}(U\cdot x)^2 \bigg\ra =  \sum_{i,j=1}^m\frac{1}{2\hone}\E \la U_iU_jx_ix_j\ra.
\end{align*}
If $i\neq j$, using the Gaussian integration by parts, we have
\begin{align*}
    \frac{1}{2\hone}\E \la U_iU_jx_ix_j\ra = \E \la x_ix_j(x_i-x'_i)(x_j+x'_j-2x''_j)\ra.
\end{align*}
If $i=j$, we have
\begin{align*}
    \frac{1}{2\hone}\E \la U_iU_ix_ix_i\ra = \E \la x_ix_i(x_i-x'_i)(x_i+x'_i-2x''_i)\ra + \frac{1}{2\hone}\E \la x^2_i\ra.
\end{align*}
These three displays combined yield
\begin{align*}
    &\E\bigg\la\frac{1}{2\hone}(U\cdot x)^2 \bigg\ra  = \E \la |x|^4 -2 |x|^2(x\cdot x')-(x\cdot x')^2 +2(x\cdot x')(x\cdot x'')\ra + \frac{1}{2\hone}\E \la |x|^2 \ra.
\end{align*}
Other terms can be computed using the Nishimori identity and the Gaussian integration by parts. We shall omit the details but only list the results:
\begin{align*}
    &\E\la (x\cdot X)^2\ra = \E\la (x\cdot x')^2\ra,\\
    &\E\bigg\la \frac{1}{\sqrt{2\hone}}(U\cdot x)(x\cdot X) \bigg\ra = \E\la |x|^2(x\cdot x')+(x\cdot x')^2 - 2(x\cdot x')(x'\cdot x'')\ra,\\
    &\E\bigg\la \frac{1}{\sqrt{2\hone}}(U\cdot x)|x|^2\bigg\ra = \E\la |x|^4 - |x|^2(x\cdot x')\ra,\\
    &\E\la (x\cdot X)|x|^2 \ra = \E\la |x|^2(x\cdot x')\ra.
\end{align*}

Inserting these computations into \eqref{eq:E(partial_H)^2} yields
\begin{align*}
    \E\la (\partial_\hone H_n)^2\ra = \E\la (x\cdot x')^2\ra + 6\E\la (x\cdot x')^2 - (x\cdot x')(x\cdot x'')\ra + \frac{1}{2\hone}\E\la|x|^2\ra.
\end{align*}
Apply the Cauchy--Schwarz inequality and the symmetry of replicas to see
\begin{align*}
    \E\la (x\cdot x')(x\cdot x'')\ra \leq \tfrac{1}{2}\E\la (x\cdot x')^2\ra+ \tfrac{1}{2}\E\la (x\cdot x'')^2\ra = \E\la (x\cdot x')^2\ra .
\end{align*}
These two displays imply \eqref{eq:key_lower_bd_approx_hj}.

\subsection{Basic estimates}

We end this section by proving basic estimates for derivatives of $F_n$ and $\bF_n$.

\begin{Lemma}\label{lemma:useful_estimates}
There is a constant $C>0$ such that the following hold for all $n\in\N$ and all $(t,h)\in\R^3_+$:
\begin{align}
    &|\bF_n|\leq C(t+|h|); \label{eq:bF<C(t+|h|)}\\
    &|\partial_t \bF_n|,\quad |\nabla \bF_n|\leq C; \label{eq:grad_bF<C} \\
    &\partial_t^2 \bF_n,\quad \partial_{h_i}\bF_n,\quad \partial_{h_i} \partial_{h_j} \bF_n \geq0,\qquad \forall  i,j \in \{1,2\};\label{eq:1st_2nd_d_bF_lower_bound}\\
    &|\partial_t F_n|\leq C\bigg(1+\frac{|W|_\mathrm{op}}{\sqrt{nt}}\bigg),\qquad  |\nabla F_n|\leq C\bigg(1+\frac{|U|+|V|}{\sqrt{n|h|}}\bigg);\label{eq:1st_der_F_n_est}\\
    &\partial^2_\hone F_n \geq -C|U|n^{-\frac{1}{2}}h_1^{-\frac{3}{2}},\qquad \partial^2_\htwo F_n \geq -C|V|n^{-\frac{1}{2}}h_2^{-\frac{3}{2}}.\label{eq:2nd_d_F_lower_bound}
\end{align}
\end{Lemma}

In \eqref{eq:1st_der_F_n_est}, we use  $|\cdot|_\mathrm{op}$ to denote the operator norm, namely,
\begin{align*}
    |W|_\mathrm{op} = \sup_{x\in \R^m,\ |x|\leq 1} \big|Wx\big|.
\end{align*}

\subsubsection*{Proof of \eqref{eq:bF<C(t+|h|)}}
Note that $\bF_n(0,0) =0$ for all $n$. Hence \eqref{eq:bF<C(t+|h|)} follows from \eqref{eq:grad_bF<C}.

\subsubsection*{Proof of \eqref{eq:grad_bF<C}}
This estimate follows from the formulae in \eqref{eq:1st_der_bF_t}-\eqref{eq:1st_der_bF_h_2} and the assumption \eqref{eq:support_X_Y} on the boundedness of $X$ and $Y$.

\subsubsection*{Proof of \eqref{eq:1st_2nd_d_bF_lower_bound}}

Due to the independence of replicas, we have $\E\la (x\cdot x')\ra =\E |\la x\ra |^2$. Hence, it is evident from \eqref{eq:1st_der_bF}-\eqref{eq:1st_der_bF_h_2} that $\partial_{h_i}\bF_n\geq 0$, for $i=1,2$.

Then, we compute $\partial_{h_i}\partial_{h_j} \bF_n$ for $i,j\in\{1,2\}$. Recall the definition of $H_n$ in \eqref{eq:H_n_expression} and expressions of first order derivatives in \eqref{eq:1st_der_bF}-\eqref{eq:1st_der_bF_h_2}. Using these and the Nishimori identity, we compute
\begin{align*}
    \partial_\hone\partial_\htwo \bF_n &= \partial_\htwo \big(N^{-1}\E \la x, x'\ra \big)\\
    &= N^{-1}\E\Big\la 2x\cdot x'\big((2\htwo)^{-\frac{1}{2}}V\cdot y+ 2 y\cdot y''-y\cdot y\big) \\
    &\qquad -2x\cdot x' \big((2\htwo)^{-\frac{1}{2}}V\cdot y'' + 2y''\cdot y''' -y''\cdot y''\big)\Big\ra.
\end{align*}
Apply the Gaussian integration to get
\begin{align*}
    \partial_\hone\partial_\htwo \bF_n &=2N^{-1}\E\Big\la x\cdot x'\big((y+y'-2y'')\cdot y+2y\cdot y''- y\cdot y\big)\\
    & \qquad -x\cdot x'\big((y+y'+y''-3y''')\cdot y'' + 2y''\cdot y'''-y''\cdot y'''\big)\Big\ra.
\end{align*}
Collecting terms and using the symmetry of replicas, we arrive at
\begin{align*}
    \partial_\hone\partial_\htwo \bF_n &= 2N^{-1}\E\big\la (x\cdot x')(y\cdot y') - 2 (x\cdot x')(y\cdot y'') + (x\cdot x')(y''\cdot y''')\big\ra\\
    & = 2N^{-1}\E\big|\la xy^\intercal \ra - \la x\ra\la y\ra^\intercal \big|^2\geq 0.
\end{align*}
To compute $\partial^2_\hone \bF_n$, we repeat the above calculation with $h_2, y, V$ replaced by $h_1, x, U$. In a similar way, we can also treat $\partial^2_\htwo \bF_n$. These calculations yield
\begin{align*}
\begin{split}
    \partial^2_\hone \bF_n &= 2N^{-1}\E\big|\la xx^\intercal \ra - \la x\ra\la x\ra^\intercal \big|^2\geq 0,\\
    \partial^2_\htwo \bF_n &= 2N^{-1}\E\big|\la yy^\intercal \ra - \la y\ra\la y\ra^\intercal \big|^2\geq 0.
\end{split}
\end{align*}

Lastly, we compute $\partial^2_t\bF_n$. Recall the formula of $\partial_t\bF_n$ in \eqref{eq:1st_der_bF_t}. Take one more derivative in $t$ and we have
\begin{align*}
    \partial^2_t\bF_n &= \frac{2}{N^3}\E\bigg\la (x\cdot x')(y\cdot y')\bigg(\sqrt{\frac{N}{2t}}x\cdot Wy+2xy^\intercal\cdot XY^\intercal -|xy^\intercal|^2 \\
    &\qquad - \sqrt{\frac{N}{2t}}x''\cdot Wy'' -2x''y''^\intercal \cdot XY^\intercal - |x''y''^\intercal|^2 \bigg)\bigg\ra.
\end{align*}
Using the Gaussian integration by parts and the Nishimori identity, we obtain, after collecting  terms,
\begin{align*}
    \partial^2_t\bF_n &= \frac{2}{N^3}\E\Big\la (x\cdot x')(y\cdot y')\Big((x\cdot x')(y\cdot y')-2(x\cdot x'')(y\cdot y'')+(x''\cdot x''')(y''\cdot y''')\Big)\Big\ra\\
    & = \frac{2}{N^3} \sum_{i,j,k,l}\E \big(\la x_iy_jx_ky_l\ra  - \la x_iy_j\ra \la x_ky_l\ra\big)^2\geq 0.
\end{align*}

\subsubsection*{Proof of \eqref{eq:1st_der_F_n_est}}
This result is a consequence of \eqref{eq:1st_der_F_t}-\eqref{eq:1st_der_F_h_2} and the boundedness assumption \eqref{eq:support_X_Y}.

\subsubsection*{Proof of \eqref{eq:2nd_d_F_lower_bound}}

Recognizing a variance term on the right hand side of \eqref{eq:2nd_der_F}, we have
\begin{align*}
    \partial^2_\hone F_n \geq - \frac{1}{N(2\hone)^\frac{3}{2}}\la U\cdot x\ra. 
\end{align*}
By the same argument, a similar lower bound also holds for $\partial^2_\htwo F_n$.

\section{Weak solutions of Hamilton--Jacobi equations}\label{section:HJ_eqn}

In this section, we study a slightly more general version of the Hamilton--Jacobi equation \eqref{eq:general_2d_HJ}. We will define the notion of weak solutions and prove the uniqueness of solutions. Under additional assumptions on the initial condition $\psi$, we verify that the Hopf formula gives a weak solution.

\smallskip

Let us describe the setting. Let $d\in\N$, $\psi:\R^d_+\to\R$ be a continuous function, and $\H:\R^d\to \R$ be smooth and satisfy
\begin{align}\label{eq:hessian_H}
    \partial^2_{ij}\H(p) \geq 0, \quad \forall p \in \R^d_+,\ \forall i,j.
\end{align}
We investigate the following equation
\begin{numcases}{}
\partial_t f(t,x) -\H\big(\nabla f(t,x)\big) =0,\quad (t,x)\in \R_+\times \R_+^d,\label{eq:general_HJ_eqn}\\
    f(0,x) = \psi(x), \quad x\in\R^d_+.\label{eq:general_initial_value}
\end{numcases}

Now, we give the precise definition of a weak solution.

\begin{Def}\label{def:weak_sol}
A function $f:\R_+\times \R_+^d$ is a weak solution to \eqref{eq:general_HJ_eqn}-\eqref{eq:general_initial_value} if
\begin{enumerate}
    \item  $f$ is Lipschitz and satisfies  \eqref{eq:general_HJ_eqn} almost everywhere;
    \item \label{item:2} $f(0,x)=\psi(x)$ for all $x\in\R^d_+$;
    \item for each $t\geq 0$, $f(t,\cdot)$ is nondecreasing;
    \item \label{item:4} for all $(t,x)\in \R_+^{d+1}$, all $\lambda\geq 0$, and all $i,j\in\{1,2,\dots,d\}$, it holds that
    \begin{align}\label{eq:partial_convex_f}
        f(t,x+\lambda e_i+\lambda e_j) +f(t,x) - f(t,x+\lambda e_i) -f(t,x+\lambda e_j) & \geq 0.
    \end{align}
\end{enumerate}
\end{Def}

Here, a function $u:\R^d_+\to \R$ is called nondecreasing provided $u(x)-u(x')\geq 0$ if $x- x'\in\R^d_+$. This monotonicity condition serves as a Neumann type boundary condition. In \eqref{item:4}, $\{e_i\}_{i=1}^d$ is the standard basis for $\R^d$. Part~\eqref{item:4} can be interpreted as a partial convexity condition. Indeed, this condition implies that, for mollifiers $\xi_\eps$ of the type introduced in Remark~\ref{Rem:infinity_norm},
\begin{align}\label{eq:2nd_mixed_derivatives_f}
    \partial^2_{ij}(f*\xi_\eps)\geq 0,\quad \forall i, j.
\end{align}

Below are our results on the uniqueness and the existence of solutions.

\begin{Prop}[Uniqueness]\label{Prop:uniq_sol_HJ}
The equation \eqref{eq:general_HJ_eqn}-\eqref{eq:general_initial_value} has at most one weak solution.
\end{Prop}

\begin{Prop}[Hopf formula]\label{Prop:hopf_formula}

Suppose $\H$ is given by \begin{align}\label{eq:def_general_H}
    \H(p)=\prod_{i=1}^dp_i,\quad p\in\R^d.
\end{align}
In addition, suppose, for $1\leq i\leq d$, there are Lipschitz, convex and nondecreasing functions $\psi_i:\R_+\to \R$ such that 
    \begin{align}\label{eq:def_general_psi}
        \psi(x)=\sum_{i=1}^d \psi_i(x_i),\quad x \in \R^d_+.
    \end{align}
Then, there exits a unique weak solution $f$ to \eqref{eq:general_HJ_eqn}--\eqref{eq:general_initial_value} given by the Hopf formula:
\begin{align*}
    f(t,x) &= \sup_{z\in \R^d_+}\inf_{y\in \R^d_+}\big\{z\cdot (x-y) +\psi(y) +t \H(z)\big\}.
\end{align*}
\end{Prop}

The conditions in Proposition~\ref{Prop:hopf_formula} might be restrictive, but are sufficient enough for our purpose. We prove the uniqueness in Section~\ref{section:uniq_sol} and show that the Hopf formula is a weak solution in Section~\ref{section:hopf_sol}. Before proceeding to proofs, let us mention a quick remark.

\begin{Rem}[Comparison principle and Lipschitz coefficients]
By a slight modification to the proof of Proposition~\ref{Prop:uniq_sol_HJ} (instead of \eqref{eq:phi_uniqueness}, let $\phi(z)=0$ for $z\leq \delta(\|f\|_\mathrm{Lip}+\|g\|_\mathrm{Lip})$ and positive otherwise), we can obtain the following comparison principle: suppose $f$ and $g$ are two weak solutions with $f(0,\cdot)\leq g(0,\cdot)$, then $f\leq g$ almost  everywhere. Using this principle and comparing $f$ with $f(\cdot, \cdot+h)$ for $h\in\R_+^2$, we can obtain
\begin{align*}
    \|f(t,\cdot)\|_\mathrm{Lip} = \|f(0,\cdot)\|_\mathrm{Lip},\quad \forall t\geq 0.
\end{align*}
Writing $\psi = f(0,\cdot)$ and using \eqref{eq:general_HJ_eqn}, we can conclude
\begin{align}\label{eq:lip}
    \|f\|_\mathrm{Lip}\leq \|\psi\|_\mathrm{Lip}\ \vee\ \bigg( \sup_{|p|\leq \|\psi\|_\mathrm{Lip}}\big|\H(p)\big|\bigg).
\end{align}
\end{Rem}

\subsection{Proof of Proposition~\ref{Prop:uniq_sol_HJ}}\label{section:uniq_sol}
The idea of this proof can be seen in \cite[Section 3.3.3]{evans2010partial}. The difference is that here $\H$ is not convex. Instead, we will utilize the fact that all entries in the Hessian of $\H$ are nonnegative as in \eqref{eq:hessian_H}, and the partial convexity \eqref{eq:2nd_mixed_derivatives_f}.

\smallskip

Let $f$ and $g$ be weak solutions to \eqref{eq:general_HJ_eqn}. Set $w = f-g$. Then we have
\begin{align*}
    \partial_t w & = \H(\nabla f)-\H(\nabla g)  = b\cdot \nabla w.
\end{align*}
where the vector $b$ is given by
\begin{align*}
    b = \int_0^1 \nabla\H\big(r \nabla f - (1-r)\nabla g\big)\d r.
\end{align*}
Take $v = \phi(w)$ for some smooth function $\phi:\R \to \R_+$ to be chosen later. Hence, we have
\begin{align}\label{eq:v_eqn}
    \partial_t v = b \cdot \nabla v.
\end{align}
To proceed, we regularize $f$, $g$ and $b$. Note that they are defined for $(t,x)\in\R_+\times \R^d_+$. Recall the mollifier $\xi_\eps$ introduced in Remark~\ref{Rem:infinity_norm}, and extend it to $\R^d$ in the obvious way. Let $f_\eps = f*\xi_\eps$, $g_\eps = g*\xi_\eps$ and
\begin{align}\label{eq:def_b_eps}
    b_\eps = \int_0^1 \nabla\H\big(r \nabla f_\eps + (1-r)\nabla g_\eps\big)\d r.
\end{align}
Here $*$ denotes the convolution in $x$. Note that these regularized versions are well-defined for $(t,x)\in\R_+\times \R^d_{\eps}$, where $\R_{\eps}=[\eps,\infty)$.

On $\R_+\times \R^d_{\eps}$, the equation \eqref{eq:v_eqn} can be expressed as
\begin{align}\label{eq:partial_t_v}
    \partial_t v = \div( vb_\eps) - v\div b_\eps +(b-b_\eps)\cdot \nabla v.
\end{align}
Before proceeding further, we need to estimate some terms in this display. 

\medskip

Due to \eqref{eq:2nd_mixed_derivatives_f} and the fact that $f$ and $g$ are nondecreasing, we have, for all $(t,x)\in\R_+\times \R^d_{\eps}$ and all $1\leq i,j\leq d$,
\begin{align*}
    \partial_if_\eps(t,x),\ \partial_ig_\eps(t,x),\ \partial^2_{ij}f_\eps(t,x),\  \partial^2_{ij}g_\eps(t,x)\geq 0.
\end{align*}
Using \eqref{eq:hessian_H} and the above display, we obtain, for $(t,x)\in\R_+\times \R^d_{\eps}$,
\begin{align}\label{eq:div_b_eps_geq_0}
    \div b_\eps =\int_0^1 \nabla^2\H\big(r \nabla f_\eps + (1-r)\nabla g_\eps\big)\cdot \big(r \nabla^2 f_\eps + (1-r)\nabla^2 g_\eps\big)\d r \geq 0.
\end{align}
Here $\nabla^2$ stands for the Hessian.
By the definitions of $f_\eps$ and $g_\eps$, we also have
\begin{align}\label{eq:grad_f_eps_bound}
    |\nabla f_\eps |\leq \|f\|_\mathrm{Lip},\quad |\nabla g_\eps |\leq \|g\|_\mathrm{Lip}.
\end{align}

Let us set
\begin{align}\label{eq:def_R_in_HJ_section}
    R&=1+ \sup\big\{|\nabla\H(p)|:p\in\R^d_+,\ |p|\leq \|f\|_\mathrm{Lip}\vee \|g\|_\mathrm{Lip}\big\}.
\end{align}
Fix any $T, \eta>0$ and define, for $t\in[0,T]$,
\begin{align}
    D_t& = \{x\in\R^d_+: |x|\leq R(T-t)\}\cap [\eta,\infty)^d,\label{eq:D_t_uniqueness}\\
    \Gamma_t &= \partial D_t \cap \{|x|=R(T-t)\}.\nonumber
\end{align}

Now, let us introduce
\begin{align*}
    J(t) = \int_{D_t}v(t,x)\d x.
\end{align*}
Let $\eps<\eta$ to ensure $D_t\subset \R^d_\eps$. Using \eqref{eq:partial_t_v} and integration by parts, we can compute 
\begin{align*}
    \frac{\d}{\d t}J(t)& = \int_{D_t}\partial_t v  - R \int_{\Gamma_t} v\\
    & = \int_{\Gamma_t}(\mathbf{n}\cdot b_\eps - R)v+ \int_{\partial D_t \setminus \Gamma_t} (\mathbf{n}\cdot b_\eps)v + \int_{D_t} v(-\div b_\eps) +\int_{D_t}(b-b_\eps)\cdot \nabla v,
\end{align*}
where $\mathbf{n}$ stands for the outer normal vector. We treat the integrals after the second equality separately. 
Due to \eqref{eq:def_b_eps}, \eqref{eq:grad_f_eps_bound} and \eqref{eq:def_R_in_HJ_section}, the first integral is nonpositive. 

For the second integral, note that on $\partial D_t \setminus \Gamma_t$, we have $-\mathbf{n}\in\R^d_+$. By \eqref{eq:def_general_H} and $f,g$ being nondecreasing, we can infer from the definition of $b_\eps$ that $b_\eps \in \R^d_+$ on $\partial D_t\setminus\Gamma_t$. 

In view of \eqref{eq:div_b_eps_geq_0}, the third integral is again nonpositive, while the last one is $o_\eps(1)$. Therefore, taking $\eps\to 0$, we conclude that
\begin{align}\label{eq:dJ_leq_0}
    \frac{\d }{\d t}J(t)\leq 0.
\end{align}

Since $w(0,x) = f(0,x)-g(0,x)=0$, for each $\delta>0$, we have $\|w(\delta, \cdot)\|_\infty \leq \delta(\|f\|_\mathrm{Lip}+\|g\|_\mathrm{Lip})$. Let us choose $\phi$ to satisfy 
\begin{align}
    \begin{cases}\label{eq:phi_uniqueness}
    \phi(z) = 0,&\quad \text{if }|z|\leq \delta(\|f\|_\mathrm{Lip}+\|g\|_\mathrm{Lip}),\\
    \phi(z)>0, &\quad \text{otherwise}.
    \end{cases}
\end{align}
Therefore, due to $v=\phi(w)$, we have
\begin{align*}
    J(\delta) = \int_{D_\delta}v(\delta,x)\d x =\int_{D_\delta}\phi(w(\delta ,x ))\d x = 0.
\end{align*}
Since $J(t)$ is nonnegative, \eqref{eq:dJ_leq_0} implies that $J(t)=0$ for all $t\in[\delta, T]$. This together with the definition of $\phi$ guarantees that
\begin{align*}
    |f(t,x)-g(t,x)|\leq \delta(\|f\|_\mathrm{Lip}+\|g\|_\mathrm{Lip}), \quad \forall x\in D_t,\ \forall t\in[\delta, T].
\end{align*}
Recall the definition of $D_t$ in \eqref{eq:D_t_uniqueness} which depends on $T$ and $\eta$.
Taking $\delta\to0$, $\eta\to0$ and $T\to \infty$, we conclude that $f=g$.

\subsection{Proof of Proposition~\ref{Prop:hopf_formula}}\label{section:hopf_sol}
Since \eqref{eq:def_general_H} satisfies \eqref{eq:hessian_H}, uniqueness is ensured by Proposition~\ref{Prop:uniq_sol_HJ}.

Let us rewrite the Hopf formula as
\begin{align}\label{eq:f_Hopf}
\begin{split}
    f(t,x) &= \sup_{z\in \R^d_+}\inf_{y\in \R^d_+}\big\{z\cdot (x-y) +\psi(y) +t \H(z)\big\}\\
    & = \sup_{z\in \R^d_+}\{z\cdot x - \psi^*(z)+t\H(z)\}\\
    & = (\psi^*-t\H)^*(x).
\end{split}
\end{align}
Here the superscript $*$ denotes the Fenchel transformation over $\R^d_+$, namely,
\begin{align}\label{eq:def_u*}
    u^*(x) = \sup_{y\in \R^d_+}\{y\cdot x - u(y)\},\quad x\in \R^d_+.
\end{align}

In the following, we verify \eqref{eq:f_Hopf} is a weak solution. Since the supremum in \eqref{eq:f_Hopf} is taken over $\R^d_+$, it is clear that $f(t,\cdot)$ is nondecreasing. We will check $f$ satisfies the following in order: initial condition, semigroup property (or dynamic programming principle), Lipschitzness, satisfying \eqref{eq:general_HJ_eqn} almost everywhere, and partial convexity.

\subsubsection{Verification of the initial condition}
The goal is to show
\begin{align}\label{eq:initial_condition}
    \psi(x) = \sup_{z\in \R^d_+}\inf_{y\in \R^d_+}\big\{z\cdot (x-y) +\psi(x) \big\}=\psi^{**}(x),\quad x\in\R^d_+.
\end{align}
Due to the assumption on $\psi$, \eqref{eq:initial_condition} follows from the following lemma.

\begin{Lemma}[Fenchel--Moreau identity]\label{lemma:biconjugate}
Let $u:\R^d_+\to (-\infty,+\infty]$ be a function not identically equal to $+\infty$. Then, $u^{**}=u$ if and only if $u$ is convex, l.s.c.\ (lower semi-continuous), and nondecreasing.
\end{Lemma}

\begin{proof}
Let $u^{**}=u$. Note that for any function $v$, we know $v^*$ is convex and l.s.c. Since now the supremum is taken over $\R^d_+$, we have $v^*$ is nondecreasing. Hence, we must have $u$ is convex, l.s.c., and nondecreasing.

\smallskip

Let $u$ be convex, l.s.c and nondecreasing, and we want to show $u^{**}=u$.
Extend the domain of $u$ to $\R^d$ by setting $u(x)=\infty$ for $x\not \in \R^d_+$. Let $\circledast$ stand for the Fenchel transformation over $\R^d$, that is,
\begin{align*}
    u^\circledast(x) = \sup_{y\in\R^d}\{y\cdot x- u(y)\}.
\end{align*}
The convexity and lower semi-continuity of $u$ yields
\begin{align}\label{eq:u_biconjugate_whole_plane}
    u(x) = u^{\circledast\circledast}(x) = \sup_{z\in\R^d}\inf_{y\in\R^d}\{z\cdot(x-y)+u(y)\} = \sup_{z\in\R^d}\{z\cdot x - u^\circledast(z)\}.
\end{align}
The goal is to show that $\R^d $ in the above display can be replaced by $\R^d_+$ whenever $x\in\R^d_+$. By our extension of $u$, we must have 
\begin{align}\label{eq:psi*_formula}
    u^\circledast(z) = \sup_{y\in\R^d_+}\{y\cdot z - u(y)\}.
\end{align}
It remains to verify that $z\in \R^d$ in \eqref{eq:u_biconjugate_whole_plane} can be replaced by $z\in \R^d_+$ whenever $x\in\R^d_+$.

We claim that 
\begin{align}\label{eq:psi*(z)=psi*(zv0)}
    u^\circledast(z)=u^\circledast(z\vee 0),\quad \forall z\in\R^d,
\end{align}
where $z\vee 0 = (z_1\vee 0, z_2\vee 0,\dots, z_d\vee 0)$. For two vectors $x,x'\in\R^d$, we write $x\geq x'$ if $x-x'\in \R^d_+$. Indeed, if \eqref{eq:psi*(z)=psi*(zv0)} is true, then due to 
\begin{align*}
    x\cdot z \leq x\cdot (z\vee 0), \quad \forall x \in \R^d_+,
\end{align*}
we have
\begin{align*}
    \sup_{z\in\R^d}\{x\cdot z - u^\circledast(z)\}\leq \sup_{z\in\R^d}\{x\cdot (z\vee0) - u^\circledast(z\vee 0)\}= \sup_{z\in\R^d_+}\{x\cdot z- u^\circledast(z)\}.
\end{align*}
From this,  \eqref{eq:u_biconjugate_whole_plane} and \eqref{eq:psi*_formula}, we can deduce that $u=u^{**}$.

\smallskip

Now, let us verify \eqref{eq:psi*(z)=psi*(zv0)}.
Let $z\in\R^d$. For $y\in\R^d_+$, we write
\begin{align*}
    \tilde y = \big(y_1\Ind{z_1\geq 0},\ y_2\Ind{z_2\geq 0},\dots,\  y_d\Ind{z_d\geq 0}\big).
\end{align*}
Using this notation, we have
\begin{align}\label{eq:z.y_leq_z.tilde_y}
    z\cdot y  \leq (z\vee 0)\cdot y  = z\cdot \tilde y,\quad \forall y\in\R^d_+.
\end{align}
Since $u$ is non-decreasing, we also have $u (y) \geq u(\tilde y)$. Using this and \eqref{eq:z.y_leq_z.tilde_y} repeatedly, we can obtain
\begin{align*}
    \sup_{y\in\R^d_+}\{z\cdot y-u(y)\}\leq \sup_{y\in\R^d_+}\{(z\vee 0)\cdot y-u(y)\}\\
     \leq \sup_{y\in\R^d_+}\{z\cdot \tilde y-u(\tilde y)\} \leq \sup_{y\in\R^d_+}\{z\cdot y-u(y)\}.
\end{align*}
Here in the last inequality we used the fact that $\{\ty: y\in \R^d_+\} \subset \R^d_+$. From the above display, we can deduce
\begin{align*}
    \sup_{y\in\R^d_+}\{z\cdot y-u(y)\} = \sup_{y\in\R^d_+}\{(z\vee 0)\cdot y-u(y)\}.
\end{align*}
This display along with \eqref{eq:psi*_formula} implies \eqref{eq:psi*(z)=psi*(zv0)}.
\end{proof}

\subsubsection{Semigroup property}

Let $f$ be given in \eqref{eq:f_Hopf}. We want to show, for all $s\geq 0$,
\begin{align*}
    f(t+s,x) &= \sup_{z\in\R^d_+}\inf_{y\in\R^d_+}\big\{z\cdot(x-y)+f(t,y) +s\H(z)\big\}.
\end{align*}
In view of the Hopf formula \eqref{eq:f_Hopf}, this is equivalent to
\begin{align}\label{eq:semigroup_property}
    \big(\psi^* - (t+s)\H\big)^* = \big((\psi^*-t\H)^{**}-s\H\big)^*.
\end{align}
From the definition of the Fenchel transform \eqref{eq:def_u*}, it can seen that, for any $u$, 
\begin{align}\label{eq:u**leq_u}
    u^{**}\leq u.
\end{align}
Since the Fenchel transform is order-reversing, \eqref{eq:u**leq_u} implies that
\begin{align}\label{eq:semigruop_one_inequality}
    \big((\psi^*-t\H)^{**}-s\H\big)^*\geq \big(\psi^* - (t+s)\H\big)^* .
\end{align}

\smallskip

To see the other direction, we use \eqref{eq:u**leq_u} to get
\begin{align*}
    \frac{s}{t+s}\psi^* + \frac{t}{t+s}\big(\psi^*-(t+s)\H\big)^{**}\leq \psi^*-t\H.
\end{align*}
For any $u$, it can be readily checked that $u^*$ is convex and lower semi-continuous. Since the supremum in the definition of $u^*$ is taken over $\R^d_+$, we can deduce that $u^*$ is non-decreasing. Therefore, taking the Fenchel transform twice in the above display and applying Lemma~\ref{lemma:biconjugate}, we have
\begin{align*}
    \frac{s}{t+s}\psi^* + \frac{t}{t+s}\big(\psi^*-(t+s)\H\big)^{**}\leq (\psi^*-t\H)^{**}.
\end{align*}
Reorder terms and then use \eqref{eq:u**leq_u} to see
\begin{align*}
    \big(\psi^*-(t+s)\H\big)^{**} - (\psi^*-t\H)^{**}\leq \frac{s}{t}\Big((\psi^*-t\H)^{**}-\psi^*\Big)\leq -s\H.
\end{align*}
This immediately gives
\begin{align*}
    \big(\psi^*-(t+s)\H\big)^{**}\leq (\psi^*-t\H)^{**}-s\H.
\end{align*}
Taking the Fenchel transform on both sides and invoking Lemma~\ref{lemma:biconjugate}, we have
\begin{align*}
     \big(\psi^* - (t+s)\H\big)^*\geq  \big((\psi^*-t\H)^{**}-s\H\big)^*.
\end{align*}
This together with \eqref{eq:semigruop_one_inequality} verifies \eqref{eq:semigroup_property}.

\subsubsection{Lipschitzness}

Since $\psi$ is Lipschitz, we have $\psi^*(z)=\infty$ outside the compact set $\{|z|\leq \|\psi\|_{\mathrm{Lip}}\}$. This together with \eqref{eq:f_Hopf} implies that for each $x\in\R^d_+$, there is $z\in \R^d_+$ with $|z|\leq \|\psi\|_\mathrm{Lip}$ such that
\begin{align*}
    f(t,x) = z\cdot x - \psi^*(z)+t\H(z).
\end{align*}
This yields that, for any $x'\in\R^d_+$,
\begin{align*}
    f(t,x) -f(t,x')\leq z\cdot(x-x') \leq \|\psi\|_\mathrm{Lip}|x-x'|.
\end{align*}
By symmetry, we conclude that $f$ is Lipschitz in $x$, and the Lipschitz coefficient is uniform in $t$.

To show Lipschitzness in $t$, we fix any $x\in\R^d_+$. Then, we have, for some $z\in\R^d_+$ with $|z|\leq \|\psi\|_\mathrm{Lip}$,
\begin{align*}
    f(t,x)&=z\cdot x-\psi^*(z)+t\H(z)\leq f(t',x) + (t-t')\H(z)\\
    &\leq f(t',x) +|t'-t|\bigg(\sup_{|z|\leq \|\psi\|_\mathrm{Lip}}|\H(z)|\bigg)\leq f(t',x)+\|\psi\|_\mathrm{Lip}^d|t'-t|.
\end{align*}
Here in the last inequality, we used the expression of $\H$ in \eqref{eq:def_general_H}. Again by symmetry, Lipschitzness in $t$ is obtained.

\subsubsection{Hopf formula satisfies \eqref{eq:general_HJ_eqn}} 
Due to Rademacher's theorem, Lipschitzness of $f$ implies that $f$ is differentiable almost everywhere. 

We want to verify that \eqref{eq:f_Hopf} satisfies \eqref{eq:general_HJ_eqn} almost everywhere. Let $(t,x)$ be a point at which $f$ is differentiable. We can assume $(t,x)$ to satisfy $t, x_1, x_2,\dots,x_d>0$, because otherwise $(t,x)$ belongs to a Lebesgue measure zero set. Since $f(t,\cdot)$ is Lipschitz, we know that outside a compact set $f^*(t,\cdot)$ is infinity. Therefore, there is $\bar z\in\R^d_+$ such that
\begin{align}\label{eq:bar_z}
    f(t,x) = \bar z\cdot x - \psi^*(\bar z) + t\H(\bar z).
\end{align}
Then using \eqref{eq:f_Hopf}, we have, for $s\geq 0$ and $h\in \R^d$ sufficiently small,
\begin{align}\label{eq:f(t,x)_upper_bound_verify_HJ}
    f(t,x) \leq f(t-s,x+h)-\bar z\cdot h+s\H(\bar z).
\end{align}
Set $s=0$ and vary $h$ to see
\begin{align}\label{eq:z=grad_f}
    \bar z = \nabla f(t,h).
\end{align}
Then, we set $h=0$ in \eqref{eq:f(t,x)_upper_bound_verify_HJ}, take $s\to 0$ and insert \eqref{eq:z=grad_f} to obtain
\begin{align*}
    \partial_t f(t,x)\leq H(\nabla f(t,h)).
\end{align*}

To verify the other direction, we use \eqref{eq:f_Hopf} and \eqref{eq:bar_z} to see that, for $s\geq 0$,
\begin{align*}
    f(t+s,x)& \geq \bar z \cdot x - \psi^*(\bar z) +(t+s)\H(\bar z)\\
    & = f(t,x )+s\H(\bar z).
\end{align*}
Send $s\to 0$ and use \eqref{eq:z=grad_f} to see
\begin{align*}
    \partial_t f(t,x) \geq \H(\nabla f(t,x)).
\end{align*}

\subsubsection{Partial convexity}

The case $i=j$ can be deduced from the convexity of $f$ which is evident from the Hopf formula \eqref{eq:f_Hopf}.

Now consider the case $i\neq j$. By relabeling, we may assume $i=1$ and $j=2$.

The Lipschitzness of $\psi$ implies that $\psi^*$ is $\infty$ outside a compact set. Due to \eqref{eq:f_Hopf}, there are $z,z'$ such that
\begin{align}\label{eq:z_z'_property}
\begin{split}
    f(t,x+\lambda e_1) &= z\cdot (x+\lambda e_1) - \psi^*(z) + t\H(z),\\
    f(t,x+\lambda e_2) &= z'\cdot (x+\lambda e_2) - \psi^*(z') + t\H(z').
\end{split}
\end{align}

Case 1: $(z_1,z_2)\leq (z'_1,z'_2)$ or $(z_1,z_2)\geq (z'_1,z'_2)$. Let us only treat the latter case. The other case can be done in an analogous way. Using \eqref{eq:f_Hopf}, we have
\begin{align*}
    f(t,x+\lambda e_1+\lambda e_2) &\geq z\cdot (x+\lambda e_1+\lambda e_2)-\psi^*(z) + t\H(z),\\
    f(t, x) &\geq z'\cdot x-\psi^*(z') + t\H(z').
\end{align*}
This along with \eqref{eq:z_z'_property} implies that the left hand side of \eqref{eq:partial_convex_f} is bounded below by
\begin{align*}
   \lambda  z\cdot  e_2 -\lambda z'\cdot e_2=\lambda (z_2 -z'_2)\geq 0.
\end{align*}

Case 2: neither $(z_1,z_2)\leq (z'_1,z'_2)$ nor $(z_1,z_2)\geq (z'_1,z'_2)$. This condition implies that
\begin{align}\label{eq:(z_1-z'_1)(z_2-z'_2)<0}
    (z_1-z'_1)(z_2-z'_2)< 0.
\end{align}
Let $\tz =(z_1,z'_2,z_3,\dots,z_d)$ and $\tz' = (z'_1, z_2,z'_3,\dots,z'_d)$. In other words, $\tz$ is obtained from $z$ through replacing $z_2$ by $z'_2$, and similarly for $\tz'$. By \eqref{eq:f_Hopf}, for each $\delta>0$, there are $y,y'\in\R^d_+$ such that
\begin{align}\label{eq:partial_convex_lower_bound_1}
\begin{split}
    f(t,x+\lambda e_1+\lambda e_2) &\geq \tz \cdot (x+\lambda e_1+\lambda e_2-y)+\psi(y) +t\H(\tz)-\delta,\\
    f(t,x) &\geq \tz' \cdot (x-y')+\psi(y') +t\H(\tz')-\delta.
\end{split}
\end{align}
Using the same construction for $\tz,\tz'$, we set
\begin{align*}
    \ty =(y_1,y'_2,y_3,\dots,y_d),\quad \ty' = (y'_1, y_2, y'_3,\dots,y'_d).
\end{align*}
Note that
\begin{align}\label{eq:tz_ty_z_y_relation}
    \tz\cdot y+\tz'\cdot y' - z\cdot \ty - z'\cdot \ty' =0.
\end{align}
From \eqref{eq:z_z'_property}, we also have
\begin{align}\label{eq:partial_convex_lower_bound_2}
\begin{split}
     f(t,x+\lambda e_1) &\leq  z \cdot (x+\lambda e_1-\ty)+\psi(\ty) +t\H(z),\\
    f(t,x+\lambda e_2) &\leq z' \cdot (x+\lambda e_2-\ty')+\psi(\ty') +t\H(z').
\end{split}
\end{align}
To lower bound the left hand side of \eqref{eq:partial_convex_f}, we start by observing that, due to \eqref{eq:tz_ty_z_y_relation},
\begin{align*}
    &\tz \cdot (x+\lambda e_1+\lambda e_2-y) + \tz' \cdot (x-y') - z \cdot (x+\lambda e_1-\ty) - z'\cdot (x+\lambda e_2-\ty')\\
    &=(\tz+\tz'-z-z')\cdot x - (\tz\cdot y+\tz'\cdot y' - z\cdot \ty - z'\cdot \ty')+\lambda(z_1+z_2'-z_1-z'_2)=0.
\end{align*}
This along with \eqref{eq:partial_convex_lower_bound_1} and \eqref{eq:partial_convex_lower_bound_2} implies that the left hand side of \eqref{eq:partial_convex_f} can be bounded below by
\begin{align*}
    \psi(y)+\psi(y')- \psi(\ty)-\psi(\ty')+t\big(\H(\tz) + \H(\tz') - \H(z) -\H(z')\big)-2\delta.
\end{align*}
From \eqref{eq:def_general_psi}, we can see
\begin{align*}
    \psi(y)+\psi(y') = \psi(\ty ) +\psi(\ty').
\end{align*}
Lastly, due to \eqref{eq:(z_1-z'_1)(z_2-z'_2)<0} and the definition of $\H$ in \eqref{eq:def_general_H}, we can compute
\begin{align*}
    \H(\tz)+\H(\tz') - \H(z)-\H(z')= -(z_1-z'_1)(z_2-z'_2)z_3\dots z_d\geq 0.
\end{align*}
The above three displays imply that the left hand side of \eqref{eq:partial_convex_f} is bounded below by $-2\delta$. The desired result follows by setting $\delta \to 0$.

\section{Convergence of the free energy}\label{section:cvg_fre_energy}

The goal is to prove Theorem~\ref{Prop:general_cvg_bF} and other convergence results to be stated. The method is similar to the one employed in Section~\ref{section:uniq_sol}.

\smallskip

Let us state the most general result in this paper. For $c\geq 0$, define $\mathcal{A}(c)$ to be the class of all functions $\psi:\R^2_+\to\R$ satisfying the following properties:
\begin{itemize}
    \item $\psi(0)=0$ and $\|\psi\|_\mathrm{Lip}\leq c$;
    \item there is a unique weak solution $f$ to the Hamilton--Jacobi equation \eqref{eq:general_2d_HJ} with initial condition $\psi$.
\end{itemize}
Recall the definitions in \eqref{eq:def_K}-\eqref{eq:def_L}.

\begin{Th}\label{thm:most_general_cvg}
Let $c\geq 0$. There is $C>0$ such that the following holds for all $M\geq 1$, all $n\in\N$ and all $\psi\in\mathcal{A}(c)$,
\begin{align*}
    &\sup_{t\in [0,M]}\int_{[0,M]^2}\big|\bF_n(t,h)-f(t,h)\big|\d h \leq CM^2\Big(L_{\psi,CM,n}+n^{-1}+ (K_{CM,n})^\frac{2}{3}+K_{CM,n}\Big),
\end{align*}
where $f$ is the unique weak solution to \eqref{eq:general_2d_HJ} with $f(0,\cdot) = \psi$.
\end{Th}

As anticipated in Remark~\ref{Rem:X_Y_not_independent}, we state the result for the case where $X$ and $Y$ are not independent. Compared with Theorem~\ref{Prop:general_cvg_bF} where the Hopf formula gives the existence of solutions, we need additional but mild assumptions to guarantee the existence of solutions in this non-independent case.

\begin{Prop}\label{prop:cvg_for_non_independent}
Under assumptions \eqref{eq:support_X_Y} and \eqref{eq:m/n_to_alpha}, suppose $\lim_{n\to\infty}K_{M,n}=0$ for all $M\geq 1$ and that there is a $\psi:\R^2_+\to\R$ such that $\lim_{n\to\infty}L_{\psi,M,n}=0$ for all $M\geq 1$. 

\smallskip

Then \eqref{eq:general_2d_HJ} admits a unique weak solution $f$ with $f(0,\cdot) =\psi$. Furthermore, there is $C>0$ such that the following holds for all $M\geq 1$ and all $n\in\N$:
\begin{align*}
    \sup_{t\in [0,M]}\int_{[0,M]^2}\big|\bF_n(t,h)-f(t,h)\big|\d h\leq CM^2\Big(L_{\psi,CM,n}+n^{-1}+K_{CM,n}\Big).
\end{align*}

\end{Prop}

The first subsection is devoted to the proof of Theorem~\ref{thm:most_general_cvg}. After that, we prove Theorem~\ref{Prop:general_cvg_bF} and Proposition~\ref{prop:cvg_for_non_independent}. Results on Hamilton--Jacobi equations in Section~\ref{section:HJ_eqn} are needed. Some argument in the proof of Theorem~\ref{thm:most_general_cvg} will be reused to prove Proposition~\ref{prop:cvg_for_non_independent}.

\subsection{Proof of Proposition~ \ref{thm:most_general_cvg}}

Let $w_n = \bF_n -f$ and
\begin{align}\label{eq:def_r_n}
    r_n = \partial_t\bF_n - (\partial_\hone\bF_n)( \partial_\htwo \bF_n).
\end{align}
Then, we have
\begin{align}\label{eq:w_n}
    \partial_t w_n = a_n\cdot \nabla \omega_n+r_n
\end{align}
where
\begin{align}\label{eq:def_a_n}
    a_n =(a_{n,1},\ a_{n,2})= (\partial_\htwo f,\ \partial_\hone \bF_n).
\end{align}
For $\delta>0$, let $\phi_\delta:\R\to [0,\infty)$ be given by
\begin{align}\label{eq:def_phi_delta}
    \phi_\delta(x)=(\delta+x^2)^\frac{1}{2},
\end{align}
which  serves as a smooth approximation of the absolute value. 
Take $v_n = \phi_\delta(w_n)$ and multiply both sides of \eqref{eq:w_n} by $\phi'_\delta(w_n)$ to see
\begin{align*}
    \partial_t v_n = a_n\cdot \nabla v_n +\phi'_\delta(w_n)r_n
\end{align*}
Recall the mollifier $\xi_\eps$ given in Remark~\ref{Rem:infinity_norm}. Let us regularize $a_n$ by setting $a^\eps_{n,i} = a_{n,i}*\xi_\eps  $, with the convolution taken in $h$. Note that $a^\eps_n$ is well-defined for $(t,h)\in\R_+\times \R^2_\eps$ where $\R_\eps = [\eps,\infty)$. For $(t,h)\in\R_+\times \R^2_\eps$, we can rewrite the above display as 
\begin{align}\label{eq:v_n_eq}
    \partial_t v_n = \div(v_na^\eps_n) - v_n\div a^\eps_n + (a_n - a^\eps_n)\cdot \nabla v_n +\phi'_\delta(w_n)r_n.
\end{align}

Let us derive a few estimates related to $a^\eps_n$. Since $\psi\in\mathcal{A}(c)$, we know that $\|\psi\|_\mathrm{Lip}\leq c$. By \eqref{eq:lip}, there is $C>0$ such that
\begin{align}\label{eq:admissible_f_lip}
    \sup_{\psi\in\mathcal{A}(c)}\|f\|_\mathrm{Lip}\leq C.
\end{align}
where $f$ is a weak solution with initial condition $f(0,\cdot)=\psi$.
By this, \eqref{eq:grad_bF<C} and \eqref{eq:def_a_n}, there is $C>0$ such that the following holds for all $n$, all $\eps \in (0,1)$ and all $ (t,h)\in\R_+\times \R^2_\eps$,
\begin{align}
    \|a_n -a_n^\eps\|_\infty&=  o_\eps(1);\label{eq:a_n-a_n^eps}\\
    \|a_n^\eps\|_\infty &\leq \|a_n\|_\infty \leq C.
\end{align}
Using \eqref{eq:1st_2nd_d_bF_lower_bound} and \eqref{item:4} in Definition~\ref{def:weak_sol}, we also have, for $ (t,h)\in\R_+\times \R^2_\eps$,
\begin{align}\label{eq:div_a_n^eps_lower_bound}
    \div a_n^\eps = \partial_\hone \partial_\htwo \big( f *\xi_\eps\big)+ \partial_\hone \partial_\htwo \big(\bF_n*\xi_\eps\big) \geq 0. 
\end{align}

Choose $R = 1+ \sup_{n,\eps}\|a_n^\eps\|_\infty$. Let $T\geq 1$ and $\eta>0$ be specified later. Consider the following sets, indexed by $t\in[0,T]$,
\begin{align}
    D_t & = \{h\in\R_+^2: h_1,h_2\geq \eta,\ |h|\leq R(T-t)\},\label{eq:def_D_t}\\
    \Gamma_t&= \{h:|h|=R(T-t)\}\cap D_t.\nonumber
\end{align}

Let us consider the object
\begin{align}\label{eq:def_J(t)}
    J_\delta(t) = \int_{D_t} v_n(t,h)\d h = \int_{D_t} \phi_\delta\big(w_n(t,h)\big)\d h.
\end{align}
Let $\eps<\eta$, which guarantees $D_t\subset \R^2_\eps$. Differentiate $J_\delta(t)$ in $t$ and use \eqref{eq:v_n_eq} to see
\begin{align*}
    &\frac{\d}{\d t}J_\delta(t) = \int_{D_t} \partial_t  v_n - R\int_{\Gamma_t}  v_n\\
    &= \int_{\Gamma_t}(a^\eps_n\cdot \mathbf{n} -R)  v_n +\int_{\partial D_t \setminus \Gamma_t}(a^\eps_n\cdot \mathbf{n})v_n  + \int_{D_t}\Big(-v_n \div  a_n^\eps + (a_n-a_n^\eps)\cdot \nabla v_n + \phi'_\delta(w_n)r_n\Big).
\end{align*}
Here in the second identity, we used integration by parts. The first integral above is nonpositive due to the choice of $R$. The second integral is also nonpositive due to the direction of $\mathbf{n}$ on $\partial D_t \setminus \Gamma_t$, and the fact that $f$ and $\bF_n$ are nondecreasing. Applying \eqref{eq:a_n-a_n^eps}--\eqref{eq:div_a_n^eps_lower_bound}  to the last integral and sending $\eps\to 0$, we obtain
\begin{align}\label{eq:dt_J_n}
    \frac{\d}{\d t}J_\delta(t)\leq  \int_{D_t}\phi'_\delta(w_n)r_n\leq \int_{D_t}|r_n|.
\end{align}
Here, in the last inequality, we used $\|\phi'_\delta\|_\infty\leq 1$ which is evident from \eqref{eq:def_phi_delta}. Lemma~\ref{lemma:HJ} gives an upper bound for $|r_n|$. Hence, we have
\begin{align}\label{eq:int_rn_up_bdd_1}
    \int_{D_t}|r_n| \leq \bigg( \frac{C}{n}\int_{D_t}\Delta \bF_n \bigg) +\bigg(\frac{1}{2}\E\int_{D_t}\big|\nabla( F_n -  \bF_n)\big|^2\bigg).
\end{align}
In view of \eqref{eq:grad_bF<C}, after integration by parts, the first term can be bounded by $CTn^{-1}$. Here and Henceforth, we absorb $R$ into $C$. To avoid heavy notation, let us write
\begin{align}\label{eq:notation_K,L}
    K=K_{RT,n},\qquad L=L_{\psi,RT,n}.
\end{align}
For the last integral in \eqref{eq:int_rn_up_bdd_1}, we will show that
\begin{align}\label{eq:derivative_concentration}
    \E\int_{D_t}\big|\nabla( F_n -  \bF_n)\big|^2\leq CT(1 +\eta^{-\frac{1}{2}})K .
\end{align}
These estimates imply that
\begin{align}\label{eq:int_r_n_bound}
    \int_{D_t}|r_n| \leq CT\Big(n^{-1}+(1 +\eta^{-\frac{1}{2}})K\Big).
\end{align}

This along with \eqref{eq:dt_J_n} implies that
\begin{align*}
    J_\delta(t) \leq J_\delta(0)+C T^2\big(n^{-1}+(1 +\eta^{-\frac{1}{2}})K\big),\quad t\in[0,T].
\end{align*}
Recall definitions \eqref{eq:def_L}, \eqref{eq:def_phi_delta} and \eqref{eq:def_J(t)}. Hence, for $t=0$, we have 
\begin{align*}
   \lim_{\delta\to0} J_\delta(0)= \int_{D_0}\big|\bF_n(0,h)-f(0,h)\big|\d h \leq CT^2 L.
\end{align*}
Sending $\delta\to0$, from the above two displays, we derive that
\begin{align*}
    \sup_{t\in[0,T]}\int_{D_t}\big|\bF_n(t,h)-f(t,h)\big|\d h\leq C T^2\big(L+n^{-1}+(1 +\eta^{-\frac{1}{2}})K\big).
\end{align*}
We want to extend the above result from integrating over $D_t$ to $\{|h|\leq R(T-t)\}$. Due to $\psi\in\mathcal{A}(c)$, we have $\psi(0)=0$, which together with $f(0,0)=\psi(0)$ and \eqref{eq:admissible_f_lip} implies  $|f(t,h)|\leq C(t+|h|)$. Using this and \eqref{eq:bF<C(t+|h|)}, we have
\begin{align*}
    \sup_{t\in[0,T]}\int_{\{|h|\leq R(T-t)\}\setminus D_t}\big|\bF_n(t,h) - f(t,h)\big|\d h\leq \int_{\{|h|\leq R(T-t)\}\setminus D_t}C T \leq CT^2\eta,
\end{align*}
Therefore, we obtain
\begin{align*}
    \sup_{t\in[0,T]}\int_{\{|h|\leq R(T-t)\}}\big|\bF_n(t,h)-f(t,h)\big|\d h\leq C T^2\big(\eta+L+n^{-1}+(1 +\eta^{-\frac{1}{2}})K\big).
\end{align*}
Let us now choose proper values for $T$ and $\delta$. Set $T= \sqrt{2}(1+R)M/R$ to ensure $[0,M]^3\subset \{(t,h):t\in[0,T],\ |h|\leq R(T-t)\}$. Inserting this $T$ and $\eta = K^\frac{2}{3}$ into the above display to see
\begin{align*}
    \sup_{t\in[0,M]}\int_{[0,M]^2}\big|\bF_n(t,h)-f(t,h)\big|\d h \leq CM^2\big(L+n^{-1}+K^\frac{2}{3}+K\big).
\end{align*}
Recall our notation \eqref{eq:notation_K,L}. This gives the desired result.

\bigskip

It remains to verify \eqref{eq:derivative_concentration}.

\subsubsection{Proof of \eqref{eq:derivative_concentration}}
Using integration by parts, we have
\begin{align}
    \int_{D_t}\big|\nabla(F_n - \bF_n)\big|^2 &= \int_{\partial D_t}(F_n-\bF_n)\nabla(F_n-\bF_n)\cdot \mathbf{n} - \int_{D_t}(F_n-\bF_n)\Delta(F_n-\bF_n)\nonumber\\
    &\leq \|F_n-\bF_n\|_{L^\infty([0,RT]^3)}\bigg(\int_{\partial D_t}\big|\nabla(F_n - \bF_n)\big|+\int_{D_t}\big|\Delta(F_n-\bF_n)\big|\bigg),\label{eq:IBP_step}
\end{align}
Let us estimate the last integral. The lower bound \eqref{eq:1st_2nd_d_bF_lower_bound} shows $\Delta \bF_n\geq 0$, and the lower bound \eqref{eq:2nd_d_F_lower_bound} implies that
\begin{align*}
    \Delta F_n + Cn^{-\frac{1}{2}}\big(\hone^{-\frac{3}{2}}|U|+\htwo^{-\frac{3}{2}}|V|\big)\geq 0.
\end{align*}
These yield
\begin{align*}
    \int_{D_t}\big|\Delta(F_n-\bF_n)\big|&\leq \int_{D_t} \big|\Delta F_n\big|+\big|\Delta\bF_n\big|\\
    &\leq \int_{D_t} \big(\Delta F_n + \Delta\bF_n\big) + \int_{D_t} 2Cn^{-\frac{1}{2}}\big(\hone^{-\frac{3}{2}}|U|+\htwo^{-\frac{3}{2}}|V|\big).
\end{align*}
Applying integration by parts to the first integral and the definition of $D_t$ to the second integral, we can see
\begin{align*}
     \int_{D_t}\big|\Delta(F_n-\bF_n)\big| \leq \int_{\partial D_t}\big|\nabla F_n\big|+\big|\nabla\bF_n\big| \d h+ CTn^{-\frac{1}{2}}\eta^{-\frac{1}{2}}\big(|U|+|V|\big).
\end{align*}
Due to \eqref{eq:grad_bF<C} and \eqref{eq:1st_der_F_n_est}, we can obtain
\begin{align*}
    \int_{\partial D_t}\big|\nabla F_n\big|+\big|\nabla\bF_n\big|\leq CT\Big(1+n^{-\frac{1}{2}}\eta^{-\frac{1}{2}}(|U|+|V|)\Big).
\end{align*}
This display also serves as a bound for the first integral in \eqref{eq:IBP_step}. Insert the above two displays into \eqref{eq:IBP_step} to get
\begin{align*}
    \int_{D_t}\big|\nabla(F_n - \bF_n)\big|^2\leq CT\|F_n-\bF_n\|_{L^\infty([0,RT]^3)}\Big(1+n^{-\frac{1}{2}}\eta^{-\frac{1}{2}}(|U|+|V|)\Big).
\end{align*}
Recall the definition of $K_{M,n}$ in \eqref{eq:def_K}. Take expectations on both sides of this inequality and invoke the Cauchy--Schwarz inequality to conclude \eqref{eq:derivative_concentration}.

\subsection{Proof of Theorem~\ref{Prop:general_cvg_bF}}

We view \eqref{eq:general_2d_HJ} as a particular case of the equation \eqref{eq:general_HJ_eqn}-\eqref{eq:general_initial_value} with $d=2$ and $\psi$ given in \eqref{eq:assumption_gen_cvg}. Proposition~\ref{Prop:uniq_sol_HJ} guarantees the uniqueness of solutions. We want to verify conditions imposed on $\psi$ in Proposition~\ref{Prop:hopf_formula} are satisfied, in order to ensure the existence and represent the solution by the Hopf formula~\eqref{eq:Hopf_2d}. This is where the assumption on the independence of $X$ and $Y$ is needed.

\smallskip

Due to the independence, we can see that, for all $n\in\N$,
\begin{align*}
    \bF_n(0,h) = \bF_n(0,\hone,0 ) +\bF_n(0,0,\htwo)
\end{align*}
By the assumption \eqref{eq:assumption_gen_cvg}, we have $\psi(h) = \psi_1(\hone) + \psi_2(\htwo)$ where
\begin{align*}
    \psi_1(\hone) =\lim_{n\to\infty}\bF_n(0,\hone,0),\qquad \psi_2(\htwo) =\lim_{n\to\infty}\bF_n(0,0,\htwo).
\end{align*}
This display along with \eqref{eq:grad_bF<C} and \eqref{eq:1st_2nd_d_bF_lower_bound} implies that $\psi_1$ and $\psi_2$ are Lipschitz, convex and nondecreasing. Therefore, this allows us to apply Proposition~\ref{Prop:hopf_formula}, which along with the uniqueness result implies the first part of Theorem~\ref{Prop:general_cvg_bF}. Due to $\bF_n(0,0)=0$, we also have $\psi(0)=0$. The second part of Theorem~\ref{Prop:general_cvg_bF} is now a direct consequence of Theorem~\ref{thm:most_general_cvg}.

\subsection{Proof of Proposition~\ref{prop:cvg_for_non_independent}}

Note that once the existence of solutions is shown, Proposition~\ref{prop:cvg_for_non_independent} easily follows from Theorem~\ref{thm:most_general_cvg} and Proposition~\ref{Prop:uniq_sol_HJ}. To obtain the existence, the plan is to first show that $\{F_n\}_{n\in\N}$ is a Cauchy sequence in the local uniform topology and then verify that the limit, denoted as $f$, is a weak solution. 

\subsubsection{The sequence $\{\bF_n\}_{n=1}^\infty$ is Cauchy}

For this, we will need the assumptions  $\lim_{n\to\infty}K_{M,n}=0$ and $\lim_{n\to\infty}L_{\psi,M,n}=0$.

We proceed similarly as in the proof of Theorem~\ref{thm:most_general_cvg}. Recall the definitions of $r_n$ in \eqref{eq:def_r_n} and $\phi_\delta$ in \eqref{eq:def_phi_delta}. Let $n,n'\in\N$. We take $w = \bF_n-\bF_{n'}$,  $v=\phi_\delta(w)$, $a= (\partial_\htwo \bF_{n'},\partial_\hone \bF_n)$ and $a^\eps = a*\xi_\eps$ where $\xi_\eps$ is the mollifier in Remark~\ref{Rem:infinity_norm}. Similar to the derivation of \eqref{eq:v_n_eq}, we have
\begin{align*}
    \partial_t v = \div(va^\eps) - v\div a^\eps +(a-a^\eps)\cdot \nabla +\phi'_\delta(w)(r_n+r_{n'}).
\end{align*}
The only difference is that we have an additional $r_{n'}$ in the last term. 

Since $\bF_n$ is Lipschitz uniformly in $n$ due to \eqref{eq:grad_bF<C}, we can fix $R= 1+\sup_{n,\eps}\|a^\eps\|_\infty$. Set $D_t$ as in \eqref{eq:def_D_t}, and similarly take
\begin{align*}
    J_\delta(t)=\int_{D_t}v(t,h)\d h = \int_{D_t}\phi_\delta(w(t,h))\d h.
\end{align*}
By similar treatments used to obtain \eqref{eq:dt_J_n}, we have
\begin{align*}
    \frac{\d}{\d t}J_\delta(t) \leq \int_{D_t}\phi'_\delta(w)(r_n+r_{n'})\leq \int_{D_t}|r_n|+|r_{n'}|.
\end{align*}
The rest follows the exact same path after $\eqref{eq:dt_J_n}$ in the proof of Theorem~\ref{thm:most_general_cvg}. The only difference is that we have more terms due to $\bF_{n'}$, but they are treated in the same way as $\bF_n$. One can see that, instead of the estimate in Theorem~\ref{Prop:general_cvg_bF}, we obtain
\begin{align*}
    \sup_{t\in[0,M]}\int_{[0,M]^2}\big|\bF_n(t,h)-\bF_{n'}(t,h)\big|\d h &\leq CM^2\Big(L_{CM,n}+n^{-1}+ (K_{CM,n})^\frac{2}{3}+K_{CM,n}\\
    & \qquad +L_{CM,n'}+(n')^{-1}+ (K_{CM,n'})^\frac{2}{3}+K_{CM,n'}\Big).
\end{align*}
Hence, by the assumptions on the decay of $K_{M,n}$ and $L_{\psi,M,n}$, we know that $\bF_n$ is Cauchy in local $L^\infty_tL^1_h$. Due to the argument in Remark~\ref{Rem:infinity_norm}, we can upgrade this to $\bF_n$ being Cauchy locally uniformly. Let us denote the limit by $f$.

\subsubsection{Verify that $f$ is a weak condition}

We check that each property listed in Definition~\ref{def:weak_sol} is satisfied by $f$.

Since $F_n$ is nondecreasing due to \eqref{eq:1st_2nd_d_bF_lower_bound} and Lipschitz uniformly in $n$ due to \eqref{eq:grad_bF<C}, we can conclude that $f$ is nondecreasing and Lipschitz. Due to $\lim_{n\to\infty}L_{\psi,M,n}=0$, we have $f(0,\cdot)=\psi$ verifying \eqref{item:2} of Definition~\ref{def:weak_sol}. By \eqref{eq:1st_2nd_d_bF_lower_bound}, property \eqref{item:4} of the definition also holds. 

It remains to verify that $f$ satisfies \eqref{eq:general_HJ_eqn} almost everywhere (a.e.). By \eqref{eq:1st_2nd_d_bF_lower_bound}, we know that, along each coordinate, both $\bF_n$ and $f$ are convex. It is well known that convexity implies convergence of derivatives at each point of differentiability. The Lipschitzness of $f$ and Rademacher's theorem imply that $f$ is differentiable a.e. Hence, we deduce that $\partial_t \bF_n - (\partial_\hone\bF_n)(\partial_\htwo \bF_n)$ converges to $\partial_t f - (\partial_\hone f)(\partial_\htwo f)$ pointwise a.e.

Since $\bF_n$ is Lipschitz uniformly in $n$, the bounded convergence theorem implies that, for any compact $B\in (0,\infty)^2$ and $t$ a.e., 
\begin{align*}
    \int_B\Big|\partial_t f - (\partial_\hone f)(\partial_\htwo f)\Big|(t,h)\d h = \lim_{n\to\infty}\int_B\Big|\partial_t \bF_n - (\partial_\hone\bF_n)(\partial_\htwo \bF_n)\Big|(t,h)\d h.
\end{align*}
We want to show the right hand side is zero. Recall the definition of $D_t$ in \eqref{eq:def_D_t}. By choosing $T$ and $\delta $ in $D_t$ suitably, we can ensure $B\subset D_t$. Then, by \eqref{eq:def_r_n}, \eqref{eq:int_r_n_bound} and $\lim_{n\to\infty}K_{M,n}=0$, we conclude that the right hand side of the above display is zero. Since $B$ and $t$ are arbitrary, we conclude that $\partial_t f - (\partial_\hone f)(\partial_\htwo f)=0$ a.e.

\section{Application to special cases}\label{section:application}

Recall the settings for the i.i.d.\ case in Section~\ref{section:iid} and the spherical case in Section~\ref{section:spherical}.
In order to apply Theorem~\ref{Prop:general_cvg_bF} to these cases, we need to identify $\psi$ in \eqref{eq:assumption_gen_cvg} and obtain estimates of $K_{M,n}$ and $L_{\psi,M,n}$ defined in \eqref{eq:def_K}-\eqref{eq:def_L}. In this section, we establish these results, which are stated below.

\begin{Lemma}\label{lemma:concentration}
In both cases, there is $C>0$ such that the following holds for all $M\geq1$ and $n\in \N$,
\begin{align*}
    K_{M,n} \leq Cn^{-\frac{1}{2}}\big( M+\sqrt{\log n}\big).
\end{align*}
\end{Lemma}

Recall $\beta(n)$ in \eqref{eq:def_beta(n)}.

\begin{Lemma}\label{lemma:cvg_initial_iid}
In the i.i.d.\ case, the function $\psi$ in \eqref{eq:assumption_gen_cvg} is given by \eqref{eq:psi_iid}.
There is $C>0$ such that, for all $M>0$ and all $n\in \N$,
\begin{align*}
    L_{\psi,M,n}\leq CM\beta(n).
\end{align*}

\end{Lemma}

\begin{Lemma}\label{lemma:cvg_initial_spherical}
In the spherical case, the function $\psi$ in \eqref{eq:assumption_gen_cvg} is given by \eqref{eq:psi_spherical}.
There is $C>0$ such that, for all $M>0$ and all $n\in \N$,
\begin{align*}
    L_{\psi,M,n}\leq CM\big(n^{-\frac{1}{2}}+\beta(n)\big).
\end{align*}

\end{Lemma}

As a consequence, Proposition~\ref{Prop:cvg_iid} and Proposition~\ref{Prop:cvg_spherical} follows from these lemmas and Theorem~\ref{Prop:general_cvg_bF}.

\subsection{Proofs of Lemma~\ref{lemma:cvg_initial_iid} and Lemma~\ref{lemma:cvg_initial_spherical}}
Due to the independence of $X$ and $Y$, we write $P^{X,Y}_n(\d x,\d y) = P^X_m(\d x)\otimes P^Y_n(\d y)$. Then, we have
\begin{align*}
    \bF_n(0,h)= \bF_n(0,\hone,0)+\bF_n(0,0,\htwo)
\end{align*}
with
\begin{align*}
    \bF_n(0,\hone,0) &= \frac{1}{N}\E\log\int e^{\sqrt{2\hone} U\cdot x +2\hone X\cdot x - \hone|x|^2}P^X_m(\d x),\\
    \bF_n(0,0,\htwo) &= \frac{1}{N}\E\log\int e^{\sqrt{2\htwo} V\cdot y +2\htwo Y\cdot y - \htwo|y|^2}P^Y_n(\d y).
\end{align*}
Recall $N=\sqrt{mn}$. By \eqref{eq:bF<C(t+|h|)}, there is $C>0$ such that
\begin{align}\label{eq:bd_phi_k,n}
    |\bF_n(0,\hone,0) |,\ |\bF_n(0,0,\htwo)|\leq C |h|,\quad \forall n\in \N.
\end{align}

\subsubsection{The i.i.d\ case}
By the entry-wise independence, we can compute
\begin{align*}
    \bF_n(0,\hone,0) = \sqrt{\frac{m(n)m(1)}{n}}\bF_1(0,\hone,0),\qquad \bF_n(0,0,\htwo)= \sqrt{\frac{n m(1)}{m(n)}}\bF_1(0,0,\htwo).
\end{align*}
Due to \eqref{eq:m/n_to_alpha}, and the above displays, we have that $\bF_n(0,h)$ converges locally uniformly to 
\begin{align*}
    \psi(h) = \sqrt{\alpha m(1)}\bF_1(0,\hone,0) + \sqrt{\alpha^{-1}m(1)}\bF_1(0,0,\htwo).
\end{align*}
Due to \eqref{eq:def_beta(n)} and \eqref{eq:bd_phi_k,n}, for any $M>0$, we have the following convergence rate estimate
\begin{align*}
    \sup_{h\in[0,M]^2}\big|\bF_n(0,h)-\psi(h)|\leq CM\beta(n).
\end{align*}
The above two displays are exactly the content of Lemma~\ref{lemma:cvg_initial_iid}.

\subsubsection{The spherical case}
As in Section~\ref{section:spherical}, we denote the uniform measure on $\sqrt{k}\mathbb{S}^{k-1}$ by $\mathscr{U}_k$.
Note that
\begin{align*}
    \sqrt{\frac{n}{m}}\bF_1(0,\hone,0) = \frac{1}{m}\E\log\int e^{\sqrt{2\hone} U\cdot x +2\hone X\cdot x - \hone|x|^2}\mathscr{U}_m(\d x)
\end{align*}
is the free energy associated with the Hamiltonian $x\mapsto\sqrt{2\hone }U\cdot x +2\hone X\cdot x -\hone|x|^2$. The following estimate can be computed using the standard interpolation method. There is $C>0$ such that, for all $n\in\N$ and all $\hone \in \R_+$,
\begin{align}\label{eq:cvg_decoupled_free_energy_spherical}
    \bigg|\sqrt{\frac{n}{m}}\bF_1(0,\hone,0) - \Big(\hone -\frac{\log(1+2\hone)}{2}\Big)\bigg| \leq C \frac{\hone}{\sqrt{m}}.
\end{align}
For completeness, we briefly sketch the key steps in the proof of \eqref{eq:cvg_decoupled_free_energy_spherical} and omit some computations. Details can be seen in \cite[Appendix A]{luneau2020high}.

\begin{proof}[Sketch of proof]
Let $\tX$ and $\tU$ be two standard Gaussian vectors in $\R^m$, independent from each other and from other randomness. Let $P^{\tX}_m$ be the law of $\tX$. Note that $\sqrt{m}\tX/|\tX|$ has the same distribution as $\mathscr{U}_m$. For $s\in[0,1]$, we introduce an interpolating Hamiltonian 
\begin{align*}
    \cH_s(\hone, x) &= \bigg(\sqrt{2\hone(1-s)m}U\cdot \frac{x}{|x|} + 2\hone(1-s)m \frac{\tX\cdot x}{|\tX||x|}-\hone(1-s)m\bigg|\frac{x}{|x|}\bigg|^2 \bigg)\\
    &\quad +\bigg(\sqrt{2\hone s}\tU\cdot x +2\hone s \tX \cdot x-\hone s|x|^2\bigg).
\end{align*}
and define, for fixed $m$ and $\hone$,
\begin{align*}
    f(s) = \frac{1}{m}\E\log\int e^{\cH_s(\hone, x)}P^{\tX}_m(\d x).
\end{align*}
It is immediate that $\sqrt{n/m}\bF_1(0,\hone,0)=f(0)$. For $s=1$, by computing a Gaussian integration, we can see $f(1)=h_1 -\frac{1}{2}\log(1+2\hone)$. Since $|f(1)-f(0)|\leq \sup_{s\in[0,1]}|f'(s)|$, the next step is to estimate $|f'(s)|$. After some computation (see \cite[(30)-(31)]{luneau2020high}), we can obtain
\begin{align*}
    |f'(s)|\leq 2\frac{\hone}{\sqrt{m}} \Big(\E\big(|\tX|-\sqrt{m}\big)^2\Big)^\frac{1}{2}.
\end{align*}
Using the standard estimate of $\tX$ concentrating near $\sqrt{m}\mathbb{S}^{m-1}$ as $m\to \infty$ (see \cite[Theorem 3.1.1]{vershynin2018high}), we can bound the above display by $C\hone /\sqrt{m}$, achieving the result.

\end{proof}

Similarly, we also have
\begin{align*}
    \bigg|\sqrt{\frac{m}{n}}\bF_1(0,0,\htwo) - \Big(\htwo -\frac{\log(1+2\htwo)}{2}\Big)\bigg| \leq C \frac{\htwo}{\sqrt{n}}.
\end{align*}
Let us set 
\begin{align*}
    \psi(h) = \alpha \Big(\hone -\frac{\log(1+2\hone)}{2}\Big) + \alpha^{-1}\Big(\htwo -\frac{\log(1+2\htwo)}{2}\Big).
\end{align*}
These displays along with \eqref{eq:cvg_decoupled_free_energy_spherical}, \eqref{eq:def_beta(n)} and \eqref{eq:bd_phi_k,n} imply that, for some $C>0$,
\begin{align*}
    \sup_{h\in[0,M]^2}\big|\bF_n(0,h)-\psi(h)| \leq CM\big( n^{-\frac{1}{2}}+\beta(n)\big).
\end{align*}
This completes the proof of Lemma~\ref{lemma:cvg_initial_spherical}.

\subsection{Proof of Lemma~\ref{lemma:concentration}}

The plan is to first obtain an  estimate of $\E e^{\lambda^2 n|F_n-\bF_n|^2}$ for small $\lambda>0$ pointwise at each $(t,h)\in [0,M]^3$. Then, we use an $\eps$-net argument to bound $\E \sup_{(t,h)\in[0,M]^3} e^{\lambda^2 n |F_n-\bF_n|}$. The desired result follows from Jensen's inequality.

\subsubsection{Pointwise estimate}

Let $(t,h)\in [0,M]^3$. 
Denote by $G=(W,U,V)$ the Gaussian vector consisting of all Gaussian random variables in $F_n$. We also write $\E_G$, $\E_X$, $\E_Y$ as the expectation integrating over $G$, $X$, $Y$, respectively. Let $\lambda>0$ be chosen later. Using H\"older's inequality, we have
\begin{align}
    \E e^{\lambda |F_n-\bF_n|} &\leq \E\Big( e^{\lambda |F_n-\E_XF_n|}e^{\lambda |\E_X F_n-\E_{X,Y}F_n|}e^{\lambda |\E_{X,Y} F_n-\E_{X,Y,G}F_n|}\Big)\nonumber\\
    &= \Big(\E e^{3\lambda |F_n-\E_XF_n|}\Big)^\frac{1}{3}\Big(\E e^{3\lambda |\E_X F_n-\E_{X,Y}F_n|}\Big)^\frac{1}{3}\Big(\E e^{3\lambda |\E_{X,Y} F_n-\E_{X,Y,G}F_n|}\Big)^\frac{1}{3}.\label{eq:exp_moment_F_n}
\end{align}

To treat the last term, we will use the Gaussian concentration inequality. 
We start by computing
\begin{align*}
    \partial_{W_{ij}}F_n = \frac{1}{N}\sqrt{\frac{2t}{N}}\la x_i y_j\ra ,\quad 
    \partial_{U_i}F_n = \frac{1}{N}\sqrt{2\hone}\la x_i\ra, \quad \partial_{V_j}F_n = \frac{1}{N}\sqrt{2\htwo}\la y_j\ra.
\end{align*}
Therefore, by \eqref{eq:support_X_Y}, we have
\begin{align*}
    |\nabla_G F_n|^2& = \sum_{i,j}|\partial_{W_{ij}}F_n|^2 + \sum_{i}|\partial_{U_i}F_n|^2+\sum_j|\partial_{V_j}F_n|^2\\
    &=\frac{2t}{N^3}\la (x\cdot x')(y\cdot y')\ra  + \frac{2\hone}{N^2}\la x\cdot x'\ra + \frac{2\htwo}{N^2}\la y\cdot y'\ra \leq CMn^{-1}.
\end{align*}
Invoking \cite[Theorem 5.5]{boucheron2013concentration}, we obtain
\begin{align}\label{eq:G_exp}
    \E_G e^{\lambda |\E_{X,Y} F_n-\E_{X,Y,G}F_n|} \leq e^{C\lambda^2Mn^{-1}}.
\end{align}

Then, we treat the first two terms in \eqref{eq:exp_moment_F_n}. In the i.i.d.\ case, due to the boundedness assumption $|X_1|,|Y_1|\leq 1$, we have
\begin{align}\label{eq:dX_iF_n}
    \big|\partial_{X_i}F_n\big|=\Big|\frac{1}{N}\Big\la \frac{2t}{N}\sum_j Y_jx_iy_j+2h_1x_i \Big\ra\Big| \leq CMn^{-1}.
\end{align}
Using the boundedness again and \cite[Theorem 6.2]{boucheron2013concentration} (see the penultimate display in its proof), we obtain
\begin{align}\label{eq:X_exp}
    \E_Xe^{\lambda |F_n - \E_X F_n|}\leq Ce^{C\lambda^2M^2n^{-1}}.
\end{align}
Applying the same argument to the second term in \eqref{eq:exp_moment_F_n}, we have
\begin{align}\label{eq:Y_exp}
    \E_Ye^{\lambda |\E_X F_n - \E_{X,Y} F_n|}\leq Ce^{C\lambda^2M^2n^{-1}}.
\end{align}

Now, let us turn to the spherical case. Since $|X|,|Y|\leq C\sqrt{n}$, using the expression of $\partial_{X_i}F_n$ in \eqref{eq:dX_iF_n}, we have
\begin{align}\label{eq:grad_X_F_n}
    |\nabla_X F_n|\leq CMn^{-\frac{1}{2}}.
\end{align}
We need Levy's inequality stated below (see \cite[Corollary 5.4]{meckes2019random}).

\begin{Lemma}
Let $f:\mathbb{S}^{n-1}\to\R$ be Lipschitz with Lipschitz constant $L$, and let $\mathcal{X}$ be distributed uniformly on $\mathbb{S}^{n-1}$. Then there are constants $C, c>0$ such that
\begin{align*}
    \P\{|f(\mathcal{X})-\E f(\mathcal{X})|\geq Lt\}\leq Ce^{-cnt^2}
\end{align*}
\end{Lemma}
Note that $X$ is uniform on $\sqrt{m}\mathbb{S}^{m-1}$. Set $\mathcal{X} = X/\sqrt{m}$. Then, \eqref{eq:grad_X_F_n} becomes $|\nabla_{\mathcal{X}}F_n| \leq CM$. Hence, the above lemma implies
\begin{align*}
    \P_X\{|F_n - \E_X F_n|\geq t\}\leq Ce^{-cnM^{-2}t^2}.
\end{align*}
Now apply \cite[Proposition 2.5.2]{vershynin2018high} to see that \eqref{eq:X_exp} holds, and so does \eqref{eq:Y_exp} via the same method.

In conclusion, \eqref{eq:exp_moment_F_n}, \eqref{eq:G_exp}, \eqref{eq:X_exp} and \eqref{eq:Y_exp}, with $\lambda$ replaced by $\lambda \sqrt{n}$ yield
\begin{align*}
    \E e^{\lambda \sqrt{n}|F_n-\bF_n|} \leq Ce^{C\lambda^2M^2}.
\end{align*}
\cite[Proposition 2.5.2]{vershynin2018high} implies that, for $\lambda$ sufficiently small,
\begin{align}\label{eq:F_n_ptw_concentration}
    \E e^{\lambda^2 n|F_n-\bF_n|^2}\leq Ce^{C\lambda^2 M^2}.
\end{align}

\subsubsection{Application of an $\eps$-net argument}

The goal is upgrade \eqref{eq:F_n_ptw_concentration} to a bound on  $\E\sup_{(t,h)\in[0,M]^3}e^{\lambda^2n|F_n-\bF_n|^2}$. The estimate \eqref{eq:1st_der_F_n_est} implies that, for $|t-t'|+|h-h'|\leq 1$,
\begin{align*}
    |F_n(t,h)- F_n(t',h')|\leq C\Big(1+n^{-\frac{1}{2}}\big(|W|_\mathrm{op}+|U|+|V|\big)\Big)\big(|t-t'|^\frac{1}{2}+|h-h'|^\frac{1}{2}\big).
\end{align*}
For $\eps\in(0,1]$, let us introduce the $\eps $-net
\begin{align*}
    A_\eps = \{\eps,2\eps, 3\eps\dots\}^3\cap [0,M]^3.
\end{align*}
Hence, for $\lambda$ small, we have
\begin{align}
    &\E\sup_{(t,h)\in[0,M]^3}e^{\lambda^2 n|F_n-\bF_n|^2}\nonumber\\
    &\leq \E \Big(\sup_{(t,h)\in A_\eps} e^{\lambda^2n|F_n-\bF_n|^2}\Big)  \exp\Big( C\lambda^2\eps\big(n+|W|^2_\mathrm{op}+|U|^2+|V|^2)\big)\Big)\nonumber\\
    &\leq \bigg(\E\sup_{(t,h)\in A_\eps}e^{2\lambda^2n|F_n-\bF_n|^2}\bigg)^\frac{1}{2}\bigg(\E \exp\Big( C\lambda^2 \eps\big(n+|W|^2_\mathrm{op}+|U|^2+|V|^2\big)\Big)\bigg)^\frac{1}{2}\label{eq:exp_sup_F_n}
\end{align}
where we used the Cauchy--Schwarz inequality in the second inequality.
Since $|A_\eps|\leq (M/\eps)^3$. Using the union bound and \eqref{eq:F_n_ptw_concentration}, we have
\begin{align}\label{eq:exp_sup_A_eps_F_n}
    \bigg(\E\sup_{(t,h)\in A_\eps}e^{2\lambda^2n|F_n-\bF_n|^2}\bigg)^\frac{1}{2} \leq C(M/\eps)^\frac{3}{2}e^{C\lambda^2 M^2}, \quad \lambda \in \R.
\end{align}

Set $\eps = C^{-1}n^{-1}$ in \eqref{eq:exp_sup_F_n} with $C$ therein, and use \eqref{eq:exp_sup_A_eps_F_n} to see
\begin{align*}
\begin{split}
    &\E\sup_{(t,h)\in[0,M]^3}e^{\lambda^2 n|F_n-\bF_n|^2}\\
    &\leq C (Mn)^\frac{3}{2} e^{C\lambda^2 M^2} \bigg(\E \exp\Big( \lambda^2\big(1+n^{-1}(|W|^2_\mathrm{op}+|U|^2+|V|^2)\big)\Big)\bigg)^\frac{1}{2}.
\end{split}
\end{align*}
We claim that, for small $\lambda>0$, 
\begin{align}\label{eq:claim_concentration}
    \E \exp\Big( \lambda^2\big(1+n^{-1}(|W|^2_\mathrm{op}+|U|^2+|V|^2)\big)\Big) \leq C.
\end{align}
This immediately gives
\begin{align*}
    \E\sup_{(t,h)\in[0,M]^3}e^{\lambda^2 n |F_n-\bF_n|^2}\leq C(Mn)^\frac{3}{2}e^{C\lambda^2 M^2} .
\end{align*}
Finally, using Jensen's inequality, we conclude that
\begin{align*}
    \E\sup_{(t,h)\in[0,M]^3}|F_n-\bF_n|^2&\leq \lambda^{-2}n^{-1}\log\Big( \E \sup_{(t,h)\in[0,M]^3}e^{\lambda^2n|F_n - \bF_n|^2}\Big)\\
    &\leq Cn^{-1}(M^2 + \log n),
\end{align*}
as desired. The proof will be complete once \eqref{eq:claim_concentration} is verified.

\subsubsection{Proof of \eqref{eq:claim_concentration}}

We want to bound exponential moments of $|W|_\mathrm{op}$, $|U|$ and $|V|$. Using the fact that $U$ and $V$ are standard Gaussian in $\R^m$ and $\R^n$, respectively, we have, for $\lambda$ small,
\begin{align}\label{eq:exp_|U|_exp_|V|}
    \E e^{\lambda^2 n^{-1}|U|^2},\  \E e^{\lambda^2 n^{-1}|V|^2} \leq C.
\end{align}

Then, we try to bound $\E e^{\lambda^2 |W|^2_\mathrm{op}}$. For $x\in \R^m$ with $|x|\leq  1$, $\{(Wx)_i\}_{i=1}^n$ is a vector in $\R^n$ consisting of independent centered Gaussian entries with variance $|x|^2\leq 1$. Hence, there are $C,c>0$ such that for $\lambda $ small
\begin{align*}
    \E e^{\lambda^2 |Wx|^2} \leq Ce^{c\lambda^2 n},
\end{align*}
which by Chebyshev's inequality implies that
\begin{align}\label{eq:|Wx|_exp}
    \P\{\lambda^2|Wx|^2\geq t\}\leq Ce^{-t + c\lambda^2n}.
\end{align}

Now let us consider a finite set $B\subset \{x\in \R^m:|x|\leq 1\}$ with the following properties:
\begin{itemize}
    \item for any distinct $x,y\in B$, we have $|x-y|> \frac{1}{2}$;
    \item for every $x\in \{x\in \R^m:|x|\leq 1\}$, there is $y\in B$ such that $|x-y|\leq \frac{1}{2}$.
\end{itemize}
By the first property, all balls with radius $\frac{1}{2}$ and centered at a point in $B$ are disjoint. Since the union of these balls is contained in a ball of radius $2$, the size of $B$ satisfies $(1/2)^m|B| \leq 2^m$. Also recall \eqref{eq:m/n_to_alpha}. Hence, there is $a>0$ such that,
\begin{align}\label{eq:|B|}
    |B|\leq 4^m\leq a^n.
\end{align}

The second property of $B$, along with the fact $|Wx-Wy|\leq |W|_\mathrm{op}|x-y|$, implies that
\begin{align*}
    |W|_\mathrm{op}=\sup_{|x|\leq 1}|Wx|\leq \sup_{x\in B}|Wx|+\frac{1}{2}|W|_\mathrm{op}.
\end{align*}
Therefore, we have $|W|_\mathrm{op}\leq 2\sup_{x\in B}|Wx|$. This along with \eqref{eq:|Wx|_exp} and \eqref{eq:|B|} yields
\begin{align*}
    \P\{e^{\lambda^2 n^{-1}|W|^2_\mathrm{op}}\geq t\}\leq Ca^n \exp\big(-c'n\log t +cn\lambda^2\big)\leq C \bigg(\frac{a e^{c\lambda^2}}{t^{c'}}\bigg)^n,
\end{align*}
for an additional constant $c'>0$.
Writing $b=(ae^{c\lambda})^\frac{1}{c'}$, we have, for $n$ large,
\begin{align*}
    \E e^{\lambda^2 n^{-1}|W|^2_\mathrm{op}} = \int_0^\infty \P\{e^{\lambda^2 n^{-1}|W|^2_\mathrm{op}}\geq t\}\d t\leq b + \int_b^\infty C\bigg(\frac{b}{t}\bigg)^{c'n}\d t=b+\frac{Cb}{c'n-1}\leq C.
\end{align*}
This and \eqref{eq:exp_|U|_exp_|V|} imply \eqref{eq:claim_concentration}.

\subsection{Discussion on the sparse model}\label{section:sparse}

Since one of the goals of this work is to convey the versatility of our approach, we briefly discuss possible modifications of our arguments to study the sparse model, described below.

\smallskip

Let $P_1$ and $P_2$ be probability distributions supported on $[-1,1]$. Consider two sequences $\{\rho_{1,n}\}_{n=1}^\infty$ and $\{\rho_{2,n}\}_{n=1}^\infty$ of real numbers in $[0,1]$. They are interpreted as the sparsity parameters. Typically, at least one of the parameters goes to zero as $n\to\infty$.  Let $X$ and $Y$ be distributed according to
\begin{align*}
    P^{X,Y}_n = \Big(\rho_{1,n}P_1 + (1-\rho_{1,n})\delta_0\Big)^{\otimes m} \otimes \Big(\rho_{2,n}P_2 + (1-\rho_{2,n})\delta_0\Big)^{\otimes n}
\end{align*}
where $\delta_0$ is the Dirac measure at $0$. In other words, $X$ and $Y$ have i.i.d.\ entries and their distributions have certain degrees of sparsity. To compensate for this sparsity of signal, it is natural to enlarge the signal-to-noise ratio correspondingly. Namely, we introduce another sequence $\{t_n\}_{n=1}^\infty$ with $\lim_{n\to\infty}t_n=\infty$. 

\smallskip

For each $n\in\N$, let $f_n$ be the weak solution to \eqref{eq:general_2d_HJ} with initial condition $\psi_n = \bF_n(0,\cdot)$. We are interested in the asymptotics of $\bF_n(t_n,\cdot) - f_n(t_n,\cdot)$ as $n\to\infty$. Recall the definition of $\mathcal{A}(c)$ introduced above Theorem~\ref{thm:most_general_cvg}. It can be checked that all $\psi_n$ belong to $\mathcal{A}(c)$ for some $c>0$. Hence, Theorem~\ref{thm:most_general_cvg} is applicable here. However, the issue is that $t_n$ now diverges to $\infty$, while the estimate in Theorem~\ref{thm:most_general_cvg} depends polynomially on $t_n$. Therefore, more effort is needed to get meaningful estimates.

\smallskip 

Since we are mostly interested in asymptotics when $h$ is small, one place to start is to choose a better region $A_n=[0,t_n]\times [0,a_n]\times[0,b_n]$ instead of blindly applying Theorem~\ref{thm:most_general_cvg} to the region $[0,t_n]\times[0,t_n]^2$.
Examining its proof, we can notice that $M$ is linked to the choice of $R$ defined after \eqref{eq:div_a_n^eps_lower_bound}. In its current form, we artificially add $1$ to the definition of $R$, but actually it can be made to have the same order as $\|\bF_n\|_\mathrm{Lip}+\|f_n\|_\mathrm{Lip}$. Even further, instead of the circular region \eqref{eq:def_D_t}, we can choose a rectangular region with two sides varying differently, the orders of which depend on $\|\partial_{h_k} \bF_n\|_\infty$ and $\|\partial_{h_k}f_n\|_\infty$. These quantities can be bounded in terms of powers of $\rho_{k,n}$. From these considerations, also taking into account the specific orders of $t_n$ and $\rho_{k,n}$, we can determine a proper $A_n$.

\smallskip

In the proof of Theorem~\ref{thm:most_general_cvg}, we used several estimates from Lemma~\ref{lemma:useful_estimates}. In the sparsity case, we can improve these estimates by obtaining decay rates  
in terms of the sparsity parameters.  Finally, after we have chosen a better region $A_n$, the quantities $K$ and $L$ in \eqref{eq:def_K}-\eqref{eq:def_L} can be redefined accordingly. Again, we can use the sparsity to improve many related estimates.

\bibliographystyle{abbrv}
\end{document}